\documentclass[10pt,reqno]{amsart}

\usepackage{amsmath}
\usepackage{amsfonts}
\usepackage{mathrsfs}
\usepackage{amssymb}
\usepackage{xypic}
\usepackage{amsthm}
\usepackage{tabularx}
\usepackage{url}
\usepackage{setspace}     \spacing{1}
\usepackage{enumitem}
\usepackage{bbm}
\usepackage{times}
\usepackage[colorlinks=true,linkcolor=blue,citecolor=blue]{hyperref}
\theoremstyle{Theorem}
\newtheorem{theorem}{Theorem} [section]
\newtheorem{alttheorem}{Theorem} 

\usepackage{tikz}
\usetikzlibrary{arrows,decorations.markings}
\usetikzlibrary{matrix}
\usetikzlibrary{graphs}
\usetikzlibrary{backgrounds}

\newtheorem{proposition}[theorem]{Proposition} 
\newtheorem{claim}[theorem]{Claim} 
\newtheorem{lemma}[theorem]{Lemma}
\newtheorem{conjecture}[theorem]{Conjecture}

\theoremstyle{definition}
\newtheorem{definition}[theorem]{Definition}

\newtheorem{example}[theorem]{Example}
\theoremstyle{remark}

\usepackage{mdframed}
\usepackage{pstricks, pst-plot, pst-node, pst-math,pst-3dplot,pstricks-add}
\newlist{hyp}{enumerate}{1} 
\setlist[hyp]{label*=\emph{(H\arabic*)},ref= {(H\arabic*)},  resume}

\newlist{enumlemma}{enumerate}{3} 
\setlist[enumlemma]{label*=\emph{(\alph*)}, ref= {(\alph*)} }

\usepackage{marginnote}

\newcommand{\Sl}{\mathrm{SL}}
\newcommand{\SL}{\mathrm{SL}}

\newcommand{\restrict}[2]{{#1}{\restriction_{{ #2}}}}
\newcommand{\td}{\tilde}
\newcommand{\wtd}{\widetilde}
\newcommand{\diff}{\mathrm{Diff}}
\newcommand{\Diff}{\diff}
\newcommand{\Fol}{\mathcal{F}}
\newcommand{\Gol}{\mathcal{G}}

\def\calL{\mathcal L}

\def\scrW{\mathcal W}

\def \RP{\R P}

\newcommand{\sm}{\smallsetminus}
\newcommand{\R}{\mathbb {R}}

\newcommand{\Z}{\mathbb {Z}}
\newcommand{\N}{\mathbb {N}}

\newcommand{\inv}{^{-1}}

\def\ae{a.e.\ }
\def\as{a.s.\ }

\def\bs{\backslash}

\newcommand{\hol}[1][\beta]{ \, \mathrm{H\ddot{o}l}^{{#1}}}
\def\Hol{\hol}
\def\Lip{\mathrm{Lip}}
\def\lip{\Lip}

\newcommand{\lieg}{\mathfrak g}
\newcommand{\lieh}{\mathfrak h}
\newcommand{\liek}{\mathfrak k}
\newcommand{\liem}{\mathfrak m}
\newcommand{\lien}{\mathfrak n}
\newcommand{\liea}{\mathfrak a}
\newcommand{\lie}{\mathrm{Lie}}
\def\Lie{\lie}
\newcommand{\liep}{\mathfrak p}
\newcommand{\lieq}{\mathfrak q}

\renewcommand{\bar}{\overline}

\title
[Invariant measures and measurable projective factors]{Invariant measures and measurable projective factors for actions of higher-rank lattices on manifolds}
\author[A.~Brown]{Aaron Brown}
\address{University of Chicago, Chicago, IL 60637, USA}
\email{awb@uchicago.edu}

\author[F.~Rodriguez Hertz]{Federico Rodriguez Hertz}
\address{Pennsylvania State University, State College, PA 16802, USA}
\email{hertz@math.psu.edu}

\author[Z.~Wang]{Zhiren Wang}
\address{Pennsylvania State University, State College, PA 16802, USA}
\email{zhirenw@psu.edu}

\long\def\symbolfootnote[#1]#2{\begingroup\def\thefootnote{\fnsymbol{footnote}}\footnote[#1]{#2}\endgroup}

\newcommand\Ts{\rule{0pt}{2.6ex}}       
\newcommand\Bs{\rule[-1.2ex]{0pt}{0pt}} 

\def\DynkinNodeSize{1.8mm}
\def\DynkinArrowLength{1.7mm}
\tikzset{
  dnode/.style={
    circle,
    inner sep=0pt,
    minimum size=\DynkinNodeSize,
    fill=white,
    draw},
  middlearrow/.style={
    decoration={markings,
      mark=at position 0.6 with
      {\draw (0:0mm) -- +(+135:\DynkinArrowLength); \draw (0:0mm) -- +(-135:\DynkinArrowLength); \draw (0:0mm) -- +(+315:.2pt); },
    },
    postaction={decorate}
  },
  leftrightarrow/.style={
    decoration={markings,
      mark=at position 0.999 with
      {
      \draw (0:0mm) -- +(+135:\DynkinArrowLength); \draw (0:0mm) -- +(-135:\DynkinArrowLength);
      },
      mark=at position 0.001 with
      {
      \draw (0:0mm) -- +(+45:\DynkinArrowLength); \draw (0:0mm) -- +(-45:\DynkinArrowLength);
      },
    },
    postaction={decorate}
  },
  sedge/.style={
  },
  dedge/.style={
    middlearrow,
    double distance=0.5mm,
  },
  tedge/.style={
    middlearrow,
    double distance=1.0mm+\pgflinewidth,
    postaction={draw}, 
  },
  infedge/.style={
    leftrightarrow,
    double distance=0.5mm,
  },
}

\def\ADYNK{
\begin{tikzpicture}[scale=\scales]
\useasboundingbox (0,-.7) rectangle (3.5,.5);
    \node[dnode,label=below:$\alpha_1$] (1) at (0,0) {};
    \node[dnode,label=below:$\alpha_2$] (2) at (1,0) {};
    \node[dnode,label=below:$\alpha_{\ell-1}$] (3) at (2.5,0) {};
    \node[dnode,label=below:$\alpha_\ell$] (4) at (3.5,0) {};
    \path (1) edge[sedge] (2)
          (2) edge[sedge,dashed, dash phase=1.4pt] (3)
          (3) edge[sedge] (4);
 \end{tikzpicture}
}
\def\BDYNK{
 \begin{tikzpicture}[scale=\scales]
\useasboundingbox (0,-.7) rectangle (3.5,.5);
\path           (2.45,0) edge[dedge] (3.55,0);
    \node[dnode,label=below:$\alpha_1$] (1) at (0,0) {};
    \node[dnode,label=below:$\alpha_2$] (2) at (1,0) {};
    \node[dnode,label=below:$\alpha_{\ell-1}$] (3) at (2.5,0) {};
    \node[dnode,label=below:$\alpha_\ell$] (4) at (3.5,0) {};

    \path (1) edge[sedge] (2)
          (2) edge[sedge,dashed, dash phase=1.4pt] (3)
          ;
\end{tikzpicture}
}

\def\CDYNK{
\begin{tikzpicture}[scale=\scales]
\useasboundingbox (0,-.7) rectangle (3.5,.5);
\path        (3.55,0) edge[dedge]    (2.45,0);

    \node[dnode,label=below:$\alpha_1$] (1) at (0,0) {};
    \node[dnode,label=below:$\alpha_2$] (2) at (1,0) {};
    \node[dnode,label=below:$\alpha_{\ell-1}$] (3) at (2.5,0) {};
    \node[dnode,label=below:$\alpha_\ell$] (4) at (3.5,0) {};

    \path (1) edge[sedge] (2)
          (2) edge[sedge,dashed, dash phase=1.4pt] (3)
          ;
\end{tikzpicture}}

\def\BCDYNK{
\begin{tikzpicture}[scale=\scales]
\useasboundingbox (0,-.7) rectangle (3.5,.5);
\path        (2.45,0) edge[dedge]    (3.55,0);
    \node[dnode,label=below:$\alpha_1$] (1) at (0,0) {};
    \node[dnode,label=below:$\alpha_2$] (2) at (1,0) {};
    \node[dnode,label=below:$\alpha_{\ell-1}$] (3) at (2.5,0) {};
    \node[dnode,label=below:$\alpha_\ell$] (4) at (3.5,0) {};

    \path (1) edge[sedge] (2)
          (2) edge[sedge,dashed, dash phase=1.4pt] (3)
          ;
\end{tikzpicture}
}

\def\DDYNKup{
\begin{tikzpicture}[scale=\scales]

\useasboundingbox (0,-1.0) rectangle (3.5,1.25);

    \node[dnode,label=below:$\alpha_1$] (1) at (0,0) {};
    \node[dnode,label=below:$\alpha_2$] (2) at (1,0) {};
    \node[dnode,label=below:$\alpha_{\ell-2}$] (4) at (2.5,0) {};
    \node[dnode,label=above:$\alpha_{\ell-1}$] (5) at (3.5,0.5) {};
    \node[dnode,label=below:$\alpha_\ell$] (6) at (3.5,-0.5) {};

    \path (1) edge[sedge] (2)
          (2) edge[sedge,dashed, dash phase=1.4pt] (4)
          (4) edge[sedge] (5)
              edge[sedge] (6);
\end{tikzpicture}
}

\def\ssizl{14em}

\def\DynkTablePosRoots{
\begin{table}[h]
\footnotesize
\def\scales{.7}
\caption{Roots systems and positive roots for classical root systems}\label{table:1}
\begin{center}
\begin{tabular}{|c|m{83pt} |l|}
\hline
& \begin{tabular} {c}Simple roots and \\Dynkin diagram \end{tabular} &   \begin{tabular} {c}Positive roots \end{tabular} 
\\    \hline

$A_\ell$&

\begin{tabular}{c}
\ADYNK\\
 $\alpha_i = e_i - e_{i+1}$  \\      {\footnotesize $  1\le i \le \ell$}\Bs
\end{tabular}
&
\begin{tabular}{m{\ssizl} l  }
$\alpha_i +   \dots + \alpha_{k} = e_i - e_{k+1} $  & {\footnotesize $ 1\le i< k\le \ell $} 
 \end{tabular}
 \\		\hline
$B_\ell$&
 \begin{tabular}{c}
\BDYNK\\
     $\alpha_i = e_i - e_{i+1},$    \\
         {\footnotesize $  1\le i \le \ell-1$};\\     $ \alpha_\ell = e_\ell$\Bs
     \end{tabular}   
&
\begin{tabular}{m{\ssizl} l  }\Ts 
 $\alpha_i +  \dots + \alpha _{k} = e_i - e_ {k+1} $ &  {\footnotesize $1\le i \le  k \le \ell-1$}\\
$\alpha_i +  \dots + \alpha _{\ell} = e_i   $&  {\footnotesize$1\le i \le \ell $}\\ 
$\alpha_i +\dots  +\alpha _{k} +    2 \alpha _{k+1} + \dots  $&  {\footnotesize $1\le i \le k < \ell$}\\
\hfill $   + 2 \alpha _\ell   = e_i + e_ {k+1}$\Bs
\end{tabular}
\\      \hline
$C_\ell$&
  \begin{tabular}{c}
\CDYNK
\\
          $\alpha_i = e_i - e_{i+1},$   \\           {\footnotesize $  1\le i \le \ell-1$};\\
      $ \alpha_\ell = 2e_\ell$\Bs
\end{tabular}
&
\begin{tabular}{m{\ssizl} l  }$\alpha_i +  \dots + \alpha _{k} = e_i - e_ {k+1}$ & {\footnotesize $1\le i \le  k \le \ell-1$}\\
$\alpha_i +\dots  +\alpha _{k} +    2 \alpha _{k+1} + \dots 
$& {\footnotesize $1\le i \le  k < \ell$}\\
\hfill $   + 2 \alpha _{\ell-1} + \alpha _\ell= e_i + e_ {k+1}$ &\\
$2\alpha_i +\dots   + 2 \alpha _{\ell-1} + \alpha _\ell= 2e_i $ & {\footnotesize $1\le i   \le \ell$}
\end{tabular}
\\		\hline
$BC_\ell$&
  \begin{tabular}{c}
\BCDYNK\\
          $\alpha_i = e_i - e_{i+1},$   \\
      {\footnotesize $  1\le i \le \ell-1$};\\
     $ \alpha_\ell = e_\ell$ \Bs
\end{tabular}
&
\begin{tabular}{m{\ssizl} l  }$\alpha_i +   \dots + \alpha _{k} = e_i - e_ {k+1}$ & {\footnotesize $1\le i \le k \le \ell-1$}\\
$\alpha_i +  \dots + \alpha _{\ell} = e_i   $& {\footnotesize $1\le i \le \ell $}\\ 
$\alpha_i +\dots  +\alpha _{k} +    2 \alpha _{k+1} + 
 \dots $ & {\footnotesize $1\le i \le k < \ell$}\\
\hfill  $  + 2 \alpha _\ell = e_i + e_ {k+1}$\\
$2\alpha_i +\dots   + 2 \alpha _{\ell-1} +2 \alpha _\ell= 2e_i $ & {\footnotesize $1\le i   \le \ell$}
\\
\end{tabular}
\\ 		\hline
$D_\ell$&
 \begin{tabular}{c}
\DDYNKup \\
          $\alpha_i = e_i - e_{i+1},$  
    \\
      {\footnotesize $  1\le i \le \ell-1$};\\
        $ \alpha_\ell = e_{\ell-1} + e_\ell$\Bs 
\end{tabular}
&
\begin{tabular}{m{\ssizl} l  }
$\alpha_i$ &$1\le i\le \ell$\Ts \\
$\alpha_i  + \dots + \alpha _{k}= e_i - e_ {k+1}$ & {\footnotesize $1\le i < k \le \ell-2$}\\
$\alpha_i +   \dots + \alpha _{\ell-2}+ \alpha _{\ell-1} = e_i - e_ {\ell}$ & {\footnotesize $1\le i  \le \ell-2$}\\
$\alpha_i +   \dots + \alpha _{\ell-2}+ \alpha _\ell = e_i + e_ {\ell}$ & {\footnotesize $1\le i  \le \ell-2$}\\
$\alpha_i +   \dots + \alpha _{\ell-1}+ \alpha _\ell = e_i + e_ {\ell-1}$ & {\footnotesize $1\le i  \le \ell-2$}\\
$\alpha_i +\dots  +\alpha _{k} +    2 \alpha _{k+1} + 
 \dots $  & {\footnotesize  $1\le i \le  k <\ell-2$}\\
\quad \quad\quad $ + 2 \alpha _{\ell-2}  + \alpha _{\ell-1} + \alpha _\ell $\\ \quad \quad\quad $ = e_i + e_ {k+1}$\Bs 
\end{tabular}
\\		 \hline
\end{tabular}
\end{center}
\label{default}
\end{table}
}

\begin{document}
 
\begin{abstract}
We consider smooth actions of lattices in higher-rank semisimple Lie groups on manifolds.  We define two numbers  $r(G)$ and $m(G)$ associated with the roots system of the Lie algebra of a Lie group $G$.   If the dimension of the manifold is smaller than $r(G)$, then we show the action preserves a Borel probability measure.   If the dimension of the manifold is at most $m(G)$, we  show there is a quasi-invariant measure on the manifold such that the action is measurable isomorphic to a relatively measure preserving action over a standard boundary action.  
\end{abstract}
\maketitle

\section{Introduction and statement of results}\label{sec:1} 
In this paper we consider lattices  $\Gamma$ in  higher-rank Lie groups $G$ acting by $C^{1+\text{H\"older}}$ diffeomorphisms on   compact manifolds.  The \emph{Zimmer program} refers to a number of questions and conjectures related to such actions.  It is expected that all such actions   are constructed from algebraic examples.  In particular, if the dimension of $M$ is smaller than the dimension of all possible algebraic actions, \emph{Zimmer's conjecture} asserts that all actions factor through the action of a finite group.  See   \cite{1608.04995} for recent solution to  (non-volume-preserving case of) Zimmer's conjecture for cocompact lattices in split, simple Lie groups.  

The main results of this paper concern actions of lattices in low dimensions. 
Most rigidity results in the literature concerning actions of lattices in low dimensions require additional hypotheses such as the  preservation of a Borel probability measure (see \cite{MR2219247,MR1946555}), strong regularity assumptions of the action (see \cite{MR1666834}), or extremely low dimensions  (see \cite{MR1198459,MR1911660, MR1703323} for actions on the circle and  \cite{MR2219247,MR1946555} for actions on surfaces.) 
Our focus in this paper is to establish the existence of an invariant measure for   actions   in moderately low dimensions and with low differentiability.
In particular, 
in  Theorem \ref{thm:invmeas} we show that if the dimension of $M$ is sufficiently small relative to  algebraic data associated to a simple Lie group $G$, then for any lattice $\Gamma\subset G$, any  $C^{1+\beta}$-action $\alpha \colon \Gamma\to \diff^{1+\beta}(M) $ preserves a Borel probability measure.  The critical dimension below which we are guaranteed an invariant probability is precisely the critical dimension in the non-volume-preserving case of  Zimmer's conjecture for split, simple Lie groups.  In the case that $\Gamma$ is cocompact, Theorem \ref{thm:invmeas} follows immediately from the main result of  \cite{1608.04995}; on the other hand,  the proof of the main result of \cite{1608.04995} uses many of the ideas used to prove  Theorem \ref{thm:invmeas}, particularly our Proposition \ref{thm:nonresonantimpliesinvariant} below.    Theorem \ref{thm:invmeas} moreover holds for actions of nonuniform lattices whereas  Zimmer's conjecture has yet to be verified for nonuniform lattices.  

The second main result, Theorem \ref{thm:measconj}, concerns actions $\alpha \colon \Gamma\to \diff^{1+\beta}(M)$ on manifolds $M$ of certain intermediate dimensions.  This range of dimensions includes examples where there   exist non-trivial (volume-preserving) actions as well as examples of actions that do not preserve any Borel probability measure.  In this case, we show that there exists a quasi-invariant measure $\mu$  on $M$ such that the action on $(M,\mu)$ is measurably isomorphic to a relatively measure-preserving extension over a standard projective action.

Given an action $\alpha \colon \Gamma\to \diff^{1+\beta}(M)$,  the key idea in both theorems is to consider the $G$-action induced by $\alpha$ on an auxiliary space which we denote by  $M^\alpha.$  We take $P\subset G$ to be a minimal parabolic subgroup and consider $P$-invariant measures $M^\alpha$.  
  This  approach should be compared with a number of papers by Nevo and Zimmer, particularly \cite{MR1720183,MR1933077}.  Nevo and Zimmer consider a  manifold with a $G$-action and $G$-stationary measure $\nu$.  $\nu$ decomposes as $\nu_0 \ast \lambda$ where $\lambda$ is  a $P$-invariant measure (See \cite[Theorem 1.4]{MR1720183} for discussion of this decomposition).  Assuming that $\lambda$ satisfies   certain technical  conditions---namely that the measure  $\lambda$ is  either \emph{$P$-mixing} in \cite{MR1720183} or that every non-trivial element of the maximal split Cartan subgroup $S\subset P$ acts ergodically in \cite{MR1933077}---it is shown that the $G$-action on  $(M,\nu)$ is a  relatively measure-preserving extension  over a  standard projective  action. 
These  technical  conditions are typically difficult to verify.  
In our argument, we exploit the constraints on the dimension of $M$ and  verify certain conditions similar to those introduced by  Nevo and Zimmer.  
For instance, the technical condition in \cite[Theorem 3]{MR1933077} that all elements of the maximal split Cartan subgroup  $S\subset P$  act ergodically  implies our  Claim \ref{pi-part} below and hence all arguments in Section \ref{sec:thisone} apply.  
In practice, it is difficult to verify such ergodicity hypotheses.

%
%

\subsection{Introduction}
Throughout we assume that $G$ is a real, connected, semisimple  Lie group with finite center and  
  $\R$-rank     at least $2$.   
By a standard construction, there is a finite cover $\td G\to G$ such that $\td G$ is the direct product of connected, almost-simple Lie groups: $$\td G = \prod G_i.$$
We take $\Gamma\subset G$ to be lattice and, writing $\td \Gamma$ for the lift of $\Gamma$ to $\td G$, we assume that for every almost-simple factor  $G_i\subset \td G$ with $\R$-rank 1,  the projection of $\td \Gamma$ to $G_i$ is  dense in $G_i$. Such a lattice will be called a \emph{higher-rank lattice}.   
This in particular includes the cases that 
\begin{enumerate}
\item $G$ has no compact factors and $\Gamma\subset G$ is irreducible, or 
\item every non-compact, almost-simple factor of $G$ 
has $\R$-rank at least $2$.  
\end{enumerate}
Below, we will study smooth actions of such groups $\Gamma$.      
As we may lift an action of $\Gamma$ to an action of $\td \Gamma$, without loss of generality we will assume for the remainder that  $G$ is 
 a direct product  $G= \prod G_i$ of almost-simple Lie groups.  

Note that $G= C\times G'$ where $C$ is the maximal connected compact normal subgroup of  $G$ and $G'$ is the maximal connected normal subgroup without compact factors.  
We remark that our main results---Theorems \ref{thm:invmeas} and \ref{thm:measconj}---are  sharpest when  
$G'$ is assumed to be simple.

Let $M$ be a compact, connected, boundaryless $C^\infty$ manifold and let $\alpha \colon \Gamma\to \diff^{1+\beta}(M)$ be an action of $\Gamma$ on $M$ by $C^{1+\beta}$ diffeomorphisms.  For notational convenience later, we assume $\alpha$ is a \emph{right} action; that is $\alpha(gh) (x)= \alpha(h) ( \alpha (g)(x))$.   
Conjecturally, 
all such actions are obtained  from families of model algebraic actions via standard constructions.
  In particular, if $\dim( M)$ is sufficiently small   so that no model algebraic actions exists, Zimmer's conjecture states that all such actions should factor through   actions of finite quotients of $\Gamma$; that is, the image $\alpha(\Gamma)$ of $\Gamma$ in $ \diff^{1+\beta}(M)$ should be finite.  Such an action 
  is said to be \emph{trivial}.  See \cite[Conjectures I, II]{MR1666834},  \cite[Conjectures 4.12,  4.14]{F11}, or \cite[Conjecture 2.4]{1608.04995} for more precise formulations.   See also \cite{1608.04995} for recent solution to  (the non-volume-preserving case of) Zimmer's conjecture for cocompact lattices in split, simple Lie groups.  
 

We recall that in dimension 1, any lattice in a higher-rank, simple Lie group with finite center acts trivially on the circle \cite{MR1703323,MR1911660}.
For certain lattices acting on surfaces, we obtain in conjunction with the main results of \cite{MR2219247} the following  complete results. 
\begin{alttheorem}[{\cite[Corollary 1.8]{MR2219247} + Theorem \ref{thm:invmeas}}]\label{thm:easycor1}
Let $S$ be a closed oriented  surface and for $n\ge 4$, let $\Gamma\subset \Sl(n, \Z)$ be a finite index subgroup.    Then every $C^{1+\beta}$ action of $\Gamma$ on $S$ is trivial.  
\end{alttheorem}

\begin{alttheorem}[{\cite[Corollary 1.7]{MR2219247} + Theorem \ref{thm:invmeas}}]\label{thm:easycor2}
Let $S$ be a closed oriented surface of genus at least 1 and for $n\ge 4$,  let $\Gamma\subset \Sl(n, \R)$ be a {nonuniform} lattice.  Then every $C^{1+\beta}$ action of $\Gamma$ on $S$ is trivial. 
\end{alttheorem} 
More generally, Theorem \ref{thm:easycor2} holds  when $\Gamma\subset G$ is a nonuniform lattice and $G$ is a connected, semisimple Lie group with finite center, no compact factors, and $r(G)\ge 3$ for the integer $r(G)$ defined below (\cite[Corollary 1.7]{MR2219247}).  In particular, the conclusion of Theorem \ref{thm:easycor2} hold for any nonuniform lattice in a higher-rank, simple Lie group $G$ with finite center  such that the restricted root system of the Lie algebra of $G$ is not of type $A_2$.  By the main results of \cite{1608.04995},    triviality of all actions on surfaces also holds for cocompact lattices in all such groups.  

Note that if  $\Gamma\subset \Sl(3, \R)$ is any lattice then there is model  real-analytic action of $\Gamma$ on a surface $S$ that  admits no invariant probability measure---namely, the right projective action of $\Gamma\subset \Sl(3,\R)$ on $\R P^2$ (or $S^2$.) Note  that any  volume form on $\R P^2$ is quasi-invariant but non invariant under this action. 
More generally, consider $G$   a semi-simple Lie group with finite center.  Let $Q\subset G$ be a parabolic subgroup and let $\Gamma\subset G$ be a lattice.  Then there is a natural right-action of  $\Gamma$  on the quotient  $Q\bs G$ preserving no  Borel probability measure   but preserving the Lebesgue measure class.  

Given the model action discussed above, we have the following conjecture, motivated by Theorems \ref{thm:easycor1} and \ref{thm:easycor2}, attributed to Polterovich in  \cite[Question 4.8]{F11}.  
\begin{conjecture}\label{conj:sphere}
Let $\Gamma\subset \Sl(3, \R)$ be a lattice.  Let $S$ be closed, connected a surface and let $\Gamma$ act on $S$ by $C^{1+\beta}$ diffeomorphisms.  Suppose there is no $\Gamma$-invariant Borel probability measure on $S$. 
 Then $S$ is either  $\RP^2$ or $S^2$; furthermore any such action is smoothly conjugate to the standard projective action. 
\end{conjecture}

\subsection{Facts from the structure   of Lie groups} To state our main results we recall some facts and definitions from the structure theory of real Lie groups.  A standard reference is \cite{MR1920389}.  
Let $G$ be a connected,  semisimple Lie group with finite center.  
As usual, write $\lieg$ for the Lie algebra of $G$.  Fix a Cartan involution  $\theta$  of $\lieg$ and write $\liek$ and $\liep$, respectively,  for the $+1$ and $-1$ eigenspaces of $\theta$.  Denote by $\liea$  the maximal abelian subalgebra of $\liep$ and by $\liem$ the centralizer of $\liea$ in $\liek$.  We let $\Sigma$ denote the set of 
restricted roots of $\lieg$ with respect to $\liea$. Note that the elements of $\Sigma$ are real linear functionals on $\liea$.  
Recall that $\dim_\R(\liea)$ is the $\R$-rank of $G$.

We choose a family of positive roots $\Sigma_+\subset \Sigma$  
and write $\Sigma_-$ for the corresponding set of negative roots.  For $\beta\in \Sigma$ write $\lieg^\beta$ for the associated root space.  Then $\lien =\bigoplus_{\beta\in \Sigma_+} \lieg^\beta$ is a nilpotent subalgebra.  
A \emph{standard parabolic subalgebra} (relative to the choice of positive roots $\Sigma_+$) is any subalgebra  of $\lieg$ containing $ \liem\oplus\liea\oplus \lien$.  
Recall $\beta\in \Sigma_+$ is a \emph{simple} (positive) \emph{root} if it is not a linear combination of other elements in $\Sigma_+$.  We denote by 
$\Pi\subset \Sigma_+$  the set of simple roots in $\Sigma_+$.  
We have that the standard parabolic subalgebras of $\lieg$ are parametrized by exclusion of simple (negative) roots: for any sub-collection $\Pi'\subset\Pi$ let 
\begin{equation}\label{eq:parabolic}\lieq_{\Pi'}= \liem\oplus \liea \oplus \bigoplus_{\beta\in \Sigma_+ \cup \mathrm{Span}(\Pi')} \lieg^\beta.\end{equation}
Then $\lieq_{\Pi'}$ is a Lie subalgebra  of $\lieg$ and all standard parabolic subalgebras of $\lieg$ are of the form $\lieq_{\Pi'}$ for some $\Pi'\subset \Pi$.  (See \cite[Proposition 7.76]{MR1920389} and, in particular, the analysis of  corresponding $\mathfrak{sl}(2, \R)$-triples,   \cite[Lemma 7.73]{MR1920389}).

Let $A,N,$ and $ K$ be the {analytic subgroups} of $G$ corresponding to $\liea, \lien$ and $\liek$.  These are closed subgroups of $G$ and $G = KAN$ is the corresponding Iwasawa decomposition of $G$.  As $G$ has finite center, $K$ is compact.      Note that the Lie   exponential $\exp\colon \lieg\to G$ restricts to  diffeomorphisms between $\liea$ and $A$ and $\lien$ and $N$.  Fixing a basis for $\liea$, we identify $A= \exp (\liea) = \R^d$.  Via this identification we often  extend linear functionals on $\liea$ to $A$.  
We write $M= C_K(\liea)$ for the centralizer of $\liea$ in $K$.  Then $P = MAN$ is the \emph{standard minimal parabolic subgroup}. 
 Since $M$ is compact, it follows that $P$ is amenable.  
 A \emph{standard parabolic subgroup} (relative to the choice of $\theta$ and $\Sigma_+$ above) is any closed subgroup  $Q\subset G$ 
 containing $P$. 
 The Lie algebra of any standard parabolic subgroup $Q$ is a standard parabolic subalgebra and the correspondence between standard parabolic subgroups and subalgebras is 1-1.

We say two restricted roots $\beta,\hat \beta\in \Sigma$ are \emph{coarsely equivalent} if there is some $c>0$ with 
$$\hat\beta = c\beta.$$
Note that $c$ takes values only  in $\{\frac 1 2, 1, 2\}$ and this occurs only if the root system $\Sigma$ has a factor of type $BC_\ell$.    
Let $\hat\Sigma$ denote the set of \emph{coarse restricted roots}; that is, the set of coarse equivalence classes of $\Sigma$.
Note that for $\xi \in \hat \Sigma$, $\lieg^\xi := \oplus _{\beta \in \xi} \lieg ^\beta$ is a nilpotent subalgebra      
 and the Lie exponential    restricts to a  diffeomorphism between $\lieg^\xi$ and  the corresponding analytic subgroup which we denote by $G^\xi$.

Let $\lieq$ denote a standard parabolic subalgebra of $\lieg$.  
Observe that if $\lieg^\beta\subset \lieq$ for some $\beta\in \Sigma$ then, from the structure of parabolic subalgebras, $\lieg^\xi\subset \lieq$ where $\xi\in \hat \Sigma$ is the coarse restricted root containing $\beta$.
A standard parabolic (proper) subalgebra $\lieq$ is \emph{maximal} if there is no subalgebra $\lieq'$ with $\lieq\subsetneq\lieq'\subsetneq\lieg$.  Note that maximal standard parabolic subalgebras are of the form $\lieq_{\Pi\sm \beta}$ for some $\beta\in \Pi$.  

\subsection{Resonant codimension and related combinatorial numbers} 

Given a standard parabolic subalgebra $\lieq$   define the \emph{resonant codimension} of $\lieq$ to be the cardinality of the set $$\{\xi \in \hat \Sigma \mid \lieg^\xi\not\subset \lieq\}.$$
 Given $G$ as above we define a combinatorial number $r(G)$ as follows. 
\begin{definition}The \emph{minimal resonant codimension} of $\lieg$, denoted $r(\lieg)$, is defined to be the minimal value of the resonant codimension of $\lieq$ as $\lieq$ varies over all (maximal)   proper  parabolic subalgebras of $\lieg$.  
\end{definition}

\begin{example}\label{ex:1}We compute $r(\lieg)$ for a number of classical real simple Lie algebras as well as simple real Lie algebras with restricted root systems of exceptional type.  
Given a simple real Lie algebra $\lieg$ the number $r(\lieg)$ is determined purely by the  restricted root system.
In particular,  we have the following.  
\begin{description}
\item [Type $A_n$] $r(\lieg)= n$.  This includes $\mathfrak{sl}(n+1,\R),$ $ \mathfrak{sl}(n+1,\mathbb{C}), $ $\mathfrak{sl}(n+1,\mathbb H)$.
\item [Type $B_n$, $C_n$, and $(BC)_n$] $r(\lieg)= 2n-1$. This includes  $\mathfrak{sp}(n,\R)$, $\mathfrak{so}(n,m)$ for $n<m$, and 
$\mathfrak{su}(n,m)$ and
$\mathfrak{sp}(n,m)$ for $n\le m$. 
\item [Type $D_n$] $r(\lieg)= 2n-2$ for $n\ge 4.$  This includes $\mathfrak{so}(n,n)$ for $n\ge 4$.  
\item [Type $E_6$] $r(\lieg)= 16$.
\item [Type $E_7$] $r(\lieg)= 27$.
\item [Type $E_8$] $r(\lieg)=57 $.
\item [Type $F_4$] $r(\lieg)= 15$.
\item [Type $G_2$] $r(\lieg)= 5$.
\end{description}
In all classical root systems $A_n,B_n, C_n,(BC)_n$ and $D_n$ the number $r(\lieg)$ corresponds to the parabolic obtained by omitting the left-most root in the standard Dynkin diagrams.  Exceptional root systems are checked by hand.  
\end{example}

Note that if $\lieg$ is non-simple then $r(\lieg)$ is $\min\{ r(\lieg_i): 1\le i \le n\}$ where $\lieg_i$ are the simple  {non-compact} factors of $\lieg$. We write $r(G) = r(\lieg)$.   Note that our number $r(G)$ grows with the rank of $G$ but not with  the dimension of the minimal algebraic actions.  In particular, we only obtain the optimal expected results in the case that $G$ is split.  

We define a second number $m(\lieg)$ associated to the Lie algebra $\lieg$ of $G$. 
\begin{definition}
Given a simple Lie algebra $\lieg$ of $\R$-rank at least $2$, define $m(\lieg)$  to be the minimal value of the resonant codimension of $\lieq$ as $\lieq$ varies over all   proper  parabolic subalgebras $\lieq$ of the form $\lieq_{\Pi\sm \{\alpha_i, \alpha_j\}}$ where $\alpha_i\neq \alpha_j$ are simple roots in $\Pi$.  If $\lieg$ has rank $1$, let $m(\lieg) = 1$.   If $\lieg= \oplus \lieg_i$ is semisimple, take $m(\lieg)$ to be the minimum of $m(\lieg_i)$ over all non-compact, simple factors $\lieg_i$ of $\lieg$.    
\end{definition}
As before, write $m(G)=m(\lieg)$.  
\begin{example}\label{ex:2}
Again, we compute the number $m(\lieg)$ for a number of classical, simple   real Lie algebras as well as simple real Lie algebras with restricted root systems of exceptional type.  
As before, given a simple real Lie algebra $\lieg$ the number $m(\lieg)$ is determined only by the  restricted root system.
\begin{description}
\item [Type $A_n$] $m(\lieg)= 2n-1$.  
\item [Type $B_n$, $C_n$, and $(BC)_n$] $m(\lieg)= 4n-4$. 
\item [Type $D_n$] $m(\lieg) = 9$ for $ n=4$; $m(\lieg)= 4n-6$ for $n\ge 5$.  
\item [Type $E_6$] $m(\lieg)= 24$.
\item [Type $E_7$] $m(\lieg)= 43$.
\item [Type $E_8$] $m(\lieg)=84 $.
\item [Type $F_4$] $m(\lieg)= 20$.
\item [Type $G_2$] $m(\lieg)= 6$.
\end{description}
In all classical root systems except $D_4$, the number $m(\lieg)$ corresponds to the parabolic subalgebra obtained by omitting the two left-most roots in the standard Dynkin diagrams.  In $D_4$, the number $m(\lieg)$ corresponds to omitting two commuting roots.  Exceptional root systems are checked by hand.  

\end{example}

As before, write $m(G) = m(\lieg)$.

\subsection{Statement of  results}
Let $G$ be as introduced above and let $\Gamma\subset G$ be a higher-rank lattice.  Recall that $\alpha$ denotes a right action of  $\Gamma$ on a compact, boundaryless manifold $M$ by $C^{1+\beta}$ diffeomorphisms.
\subsubsection{Existence of invariant measures in  low dimensions}
Our first  main result establishes the existence of an $\alpha$-invariant measure  if the dimension $M$ is sufficiently small relative to $r(G)$.
\begin{theorem}\label{thm:invmeas}
Let $M$ be a manifold with $\dim(M)< r(G)$.  Then for any $C^{1+\beta}$ action $\alpha$ of $\Gamma$ on $M$ there exists an $\alpha$-invariant Borel probability measure.  
\end{theorem}

We remark in the case that $\Gamma$ is cocompact,  Theorem \ref{thm:invmeas} in an immediate corollary of the main result of \cite{1608.04995} where Zimmer's conjecture is verified for actions of compact lattices on manifolds of dimension less than $r(G)$.  The proof of the main result of \cite{1608.04995} uses the proof of Theorem \ref{thm:invmeas}, namely the key observation in Proposition \ref{thm:nonresonantimpliesinvariant} below.    We note also that Theorem \ref{thm:invmeas} applies to nonuniform lattices whereas Zimmer's conjecture has yet to be verified for nonuniform lattices.  
We do not assert any regularity of the measure in Theorem \ref {thm:invmeas}.  In particular, the ergodic components of the measure are expected  to be supported on finite sets as such actions are expected to be trivial.    
Theorems \ref{thm:easycor1} and \ref{thm:easycor2} follow directly from   the main results in \cite{MR2219247} and Theorem \ref{thm:invmeas}.

\subsubsection{Finite extensions  of  projective factors in critical dimension}
In the  case where $\dim M = r(G)$, we recall as a model the standard right action of $\Gamma\subset \Sl(n+1, \R)$ on $\RP^{n}$.  Note that  $\RP^{n}$ has the structure of $Q\bs \SL(n+1, \R)$ for a (maximal) parabolic subgroup $Q \subset \Sl(n+1, \R)$.


\begin{theorem}\label{thm:measconjfin}
Let $M$ be a manifold with $\dim(M)= r(G) $.  Then given any $C^{1+\beta}$ action $\alpha$  of $\Gamma$ on $M$ either
\begin{enumlemma}
\item  there exists an $\alpha$-invariant Borel probability measure on $M$; or
\item there exists an $\alpha$-quasi-invariant Borel probability measure $\mu$ on $M$ and a maximal parabolic subgroup $Q\subset G$ such that the action $\alpha$ of $\Gamma$ on $ (M, \mu)$ is measurably conjugate to a finite extension of the standard right action of $\Gamma$ on $(Q\bs G, m)$ where $m$ is of Lebesgue class.   
\end{enumlemma}
\end{theorem}
Motivated by the above theorem, we   extend Conjecture \ref{conj:sphere}.
\begin{conjecture}
Let $M$ be a manifold with $\dim(M)= r(G) $.  Given any sufficiently smooth action $\alpha$  of $\Gamma$ on $M$ either
\begin{enumlemma}
\item  there exists an $\alpha$-invariant Borel probability measure on $M$; or
\item there is  a maximal parabolic subgroup $Q\subset G$ such that 
$M$ is diffeomorphic to a  finite cover of $Q\bs G$; moreover, the action $\alpha$ is smoothly conjugate to a lift of   the standard right-action of $\Gamma$ on $Q\bs G$.  
\end{enumlemma}

\end{conjecture}

\subsubsection{Projective factors in intermediate  dimensions}

\def\aut{\mathrm{Aut}}
Let $(X,\nu)$ and $(Z,\mu)$ be  standard measure spaces and suppose $\Gamma$ acts measurably on both $X$ and $Z$ (on the right) and preserves  the measure classes of $\nu$ and $\mu$ respectively.  Let $(Y, \eta)$ be a standard measure space and write $\aut(Y,\eta)$ for the group of invertible, measure-preserving transformations of $(Y,\eta)$.  Let $\alpha$ and $\rho$ denote, respectively, the actions of $\Gamma$ on $(Z,\mu)$ and $(X,\nu)$.
\begin{definition}
We say $\alpha$ is a \emph{relatively measure-preserving extension} (modeled on $(Y,\eta)$) of $\rho$ if there  are
\begin{enumerate}
	\item a measurable cocycle $\psi\colon \Gamma\times (X,\nu)\to \aut (Y, \eta)$ over $ \rho$, and 
	\item an isomorphism of measure spaces $\Phi\colon (Z,\mu) \to (X\times Y, \nu\times \eta)$
\end{enumerate}
such that $\Phi$ intertwines $\alpha$ and the skew action defined by $\psi$: if $\Phi(z) = (x,y)$ then 
$$\Phi( \alpha(\gamma) (z))= \left(\rho(\gamma)(x), \psi(\gamma,x)(y)\right).$$
\end{definition}

\begin{theorem}\label{thm:measconj}
Let $M$ be a manifold with $\dim(M)\le m(G) $.  Then given any $C^{1+\beta}$ action $\alpha$  of $\Gamma$ on $M$ there is an $\alpha$-quasi-invariant Borel probability measure $\mu$ on $M$, a standard parabolic subgroup $Q$, and a Lebesgue space $(Y,\eta)$ such that 
the action $\alpha$ on $(M, \mu)$ is a relatively measure-preserving extension (modeled on $(Y, \eta)$) of the standard right action of $\Gamma$ on $(Q\bs G, m)$.  
\end{theorem}
Note in the above  theorem, if $Q= G$ it follows that $\mu$ is $\alpha$-invariant.  
As discussed above, the result in Theorem \ref{thm:measconj} should be compared to results of Nevo and Zimmer, particularly \cite{MR1720183,MR1933077}.  

\section{Suspension construction and its properties} 
\label{sec:2}
We construct an auxiliary   space  on which the  action $\alpha$ of $\Gamma$ on $M$ embeds as a Poincar\' e section for an associated $G$-action.  
On the product $G\times M$ consider the right $\Gamma$-action
	$$ (g,x)\cdot \gamma= (g\gamma , \alpha(\gamma)(x))$$ and the left $G$-action $$a\cdot (g,x) = (ag, x).$$
Define  the quotient manifold $M^\alpha:= G\times M/\Gamma $.  As the  $G$-action on $G\times M$ commutes with the $\Gamma$-action, we have an induced left $G$-action  on $M^\alpha$.  We denote this action by $\td \alpha$.  
We write $\pi\colon M^\alpha\to G/\Gamma$ for the natural projection map.  Note that $M^\alpha$ has the structure of a fiber bundle over $G/\Gamma$ induced by the map $\pi$ with fibers diffeomorphic to $M$.  As the action of $\alpha$ is by $C^{1+\beta}$ diffeomorphism, $M^\alpha$ is naturally a  $C^{1+\beta}$ manifold.  Equip $M^{\alpha}$ with a $C^\infty$ structure compatible with the $C^{1+\beta}$-structure.  

Note that the action $\td \alpha$ of $G$ on $M^\alpha$ preserves two transverse distributions $E^F$ and $E^G$ where $E^F= \mathrm{ker}(D\pi)$ and $E^G$ is   tangent   to the local $G$-orbits on  $M^\alpha$.  Furthermore, these distributions integrate to $C^{1+\beta}$ foliations of $M^\alpha$.  

We first observe
\begin{claim}\label{claim:measureupstairs} There exists an $\alpha$-invariant Borel probability measure on $M$ if and only if there exists an  $\td \alpha$-invariant  Borel probability measure on $M^\alpha$.  
\end{claim}
\noindent That an $\alpha$-invariant    measure on $M$ induces an $\td \alpha$-invariant      measure on $M^\alpha$ is standard.  
For the reverse implication, see, for instance, \cite[Lemma 6.1]{MR1720183}.  
Note that any $\td\alpha$-invariant measure on $M^\alpha$ projects under $\pi$  to the Haar measure on $G/\Gamma$.   

As the suspension space $M^\alpha$ is   non-compact in the case that $\Gamma$ is non-uniform, some care is needed when applying tools from smooth ergodic theory to the $G$-action on $M^\alpha$.  Indeed, although  the non-compactness comes from  the homogeneous factor, care is needed in order to control the fiber-wise dynamics as the corresponding $C^1$- and $C^{1+\beta}$-norms of the fiberwise dynamics need not be bounded.

Below, we use the quasi-isometry between the Riemannian and word metrics on $\Gamma$ established in   \cite{MR1828742} to  control the degeneration of the fiber-wise dynamics.  
We follow the approach of \cite{AWB-GLY-P1} and construct dynamical charts relative to which the tools of classical smooth ergodic theory   
 may be applied.  The remainder of this section is devoted to constructing a Riemannian metric on $TM^\alpha$, corresponding distance function $d$, and a family of dynamical charts.   

The reader interested only in actions of cocompact lattices may skip the remainder of this section.  

\subsection{Construction of a fundamental domain and family of fiber metrics} \label{sec:FD} 
Recall our standing assumptions on the Lie group $G$ and the lattice $\Gamma$.  
A set $D\subset G$ is a \emph{fundamental domain} for $\Gamma$ if $\bigcup_{\gamma\in \Gamma} D\gamma = G$ and if the natural map $G\to G/\Gamma$ is one-to-one on $D$.  
A Borel set $D\subset G$ is \emph{almost-open} if the interior of $D$ has full measure in the closure of $D$.  
$S\subset G$ is a \emph{fundamental set} if  $\bigcup_{\gamma\in \Gamma} S\gamma = G$ and  the set $\{ \gamma : S\gamma \cap S\neq \emptyset\}$ is finite.
The injectivity radius $r^\Gamma(g)$ of $\Gamma$ at a point $g\in G$ is the largest $0<r\le 1$ such that the map $\lieg\to G/\Gamma$ given by  $X\mapsto \exp_{\lieg}(X) g\Gamma$ is injective on $$\{ X\in \lieg : \| X\|<r\}.$$  
We write $$V_r(g) := \{ \exp_{\lieg}(X) g\Gamma : \|X\|\le r\}$$ for the remainder.  

Our goal below is to build on $TM$ a 
 family of continuous Riemannian metrics $\langle \cdot , \cdot \rangle_g$, parameterized by $g\in G$, and an almost-open, Borel fundamental domain $D\subset G$  for $\Gamma$ such that  
\begin{enumerate}
\item the family of metrics $\langle \cdot , \cdot \rangle_g$ depends continuously 
on $g\in G$;
\item
the family   $\langle \cdot , \cdot \rangle_g$  is \emph{$\Gamma$-equivariant}: given $\gamma\in \Gamma$ and $v,w \in T_x M$
$$\langle v , w \rangle_g = \langle D_x \alpha (\gamma) v  , D_x \alpha (\gamma) w \rangle_{g\gamma};$$
\item writing $$V= \bigcup _{g\in D} V_{r^\Gamma(g)} (g),$$   the family $\langle \cdot , \cdot \rangle_g$ is uniformly comparable on $V$: there is a $C>0$ so that for all $g, \bar g\in V, x\in M,$ and $v,w\in T_xM$ 
$$ \langle v , w \rangle_g \le C \langle v , w \rangle_{\bar g};$$
\item for every $p\ge 1$ the function $g\mapsto d_G(e,g)$ is $L^p$ on $D$ with respect to the Haar measure where $d_G(\cdot,\cdot)$  is the right-invariant metric on $G$.  
\end{enumerate}

Note that given a finite-index subgroup $\Gamma '\subset \Gamma$, a fundamental domain $D'$ for $\Gamma'$  and a $\Gamma'$-equivariant family of metrics which satisfy (1)--(4) above for $\Gamma'$, then we can choose a fundamental domain $D\subset D'$ for $\Gamma$ and construct a $\Gamma$-equivariant family of metrics satisfying (1)--(4) for $\Gamma$.  Below, we will pass to a finite-index subgroup $\Gamma '\subset \Gamma$ and  construct such a domain and family of metrics for $\Gamma'$.

To build such a family, first note that by quotienting by any  compact factors and the center of $G$, we obtain a surjective homomorphism with compact kernel $\Psi\colon G\to \bar G$ where $\bar G$ is semisimple, without any compact factors, and has  trivial center.  Moreover, the image $\bar \Gamma:= \Psi(\Gamma)$ is a lattice in $\bar G$. 
From the Margulis Arithmeticity Theorem \cite{MR1090825}, it follows that there is a semisimple, linear algebraic group ${\bf H}$ such that setting 
 $H= {\bf H}(\R)^\circ$ there is a surjective homomorphism $\Phi \colon H^\circ \to \bar G$ with compact kernel  such that $$\bar \Gamma\cap \Phi ({\bf H} (\Z))$$
has finite index in $\bar \Gamma$.  
Let $\hat \Gamma = \Phi\inv (\bar \Gamma)\cap {\bf H}(\Z)\cap H$.  Then $\hat \Gamma$ has finite index in ${\bf H}(\Z)$ and is hence arithmetic.   Replacing $\Gamma$ and $\hat \Gamma$ with  finite index subgroups,   we may assume that $\hat \Gamma$ is torsion-free and {neat} (see \cite{MR2189882} for definition), and that $\Gamma$ and $\hat \Gamma$ map surjectively onto $\bar \Gamma$.  Let $N$ denote the kernel of $\Psi\colon \Gamma\to \bar \Gamma$.  Note that $N$ is finite. 

We will use $\bar \Gamma$ to  build  a family of metrics and domain $D$.  Note that while $\bar \Gamma$ may not act on $M$, $\bar \Gamma$ acts on the space of $N$-invariant Riemannian metrics on $TM$.

\subsubsection{Compactification of  $H/\hat \Gamma$}
Let $\hat K\subset H$ be a maximal compact subgroup.  We may assume that $\hat K$ contains the kernel of the map $\Phi\colon H\to \bar G$.  Let $X$ denote the symmetric space $\hat K \bs H$.  
Following \cite[III, Chapter 9]{MR2189882}, write  $\overline X^{BS}$ for  the \emph{Borel-Serre partial compactification} of $X$.   $\overline X^{BS}$ has the structure of a real-analytic manifold with corners.  The action of $\hat \Gamma $ on $X$ extends to a continuous, proper action on $\overline X^{BS}$.  Moreover the quotient $\overline X^{BS}/\hat \Gamma $ is a compact, Hausdorff space.  
Furthermore (having taken  $\hat \Gamma $ to be neat) the quotient $\overline X^{BS}/\hat \Gamma $ has the structure of a real-analytic manifold with corners.  

\subsubsection{Parameterized families of metrics} As $\hat \Gamma$ maps surjectively onto $\bar \Gamma$, it follows that $\hat \Gamma$ acts on the space of $N$-invariant Riemannian metrics on $M$.  Consider $\overline X^{BS} \times TM$.  
As the manifold with corners $\overline X^{BS}/\hat \Gamma $ admits partitions of unity, by selecting a $\hat \Gamma$-equivariant partition of unity on $\overline X^{BS} $ subordinate to a cover by sets of the form $V_{r^{\hat \Gamma}(g)}(g)$, given any fixed $N$-invariant metric on $TM$ we may construct a $\hat \Gamma$-equivariant, continuous family of $N$-invariant, H\"older continuous Riemannian metrics $\langle \cdot, \cdot\rangle_x$ on $TM$ parametrized by points of $x\in \overline X^{BS}$.  
As $\hat K$  contains   the kernel of $\Phi$, given $\bar g\in \bar G$ we associate a Riemannian metric on $TM$ by $$\langle \cdot, \cdot\rangle_{\bar g} =   \langle \cdot, \cdot\rangle_{\Phi\inv (\bar g)}. $$ 
The family $\{\langle \cdot, \cdot\rangle_{\bar g}: \bar g\in \bar G\}$ is continuous in $\bar g$ and is $\bar \Gamma$-equivariant.  
Finally, to  $g\in G$ we associate a Riemannian metric on $TM$ by $$\langle \cdot, \cdot\rangle_{g} =   \langle \cdot, \cdot\rangle_{\Psi (g)}. $$
The family $\{ \langle \cdot, \cdot\rangle_{g} : g\in G\}$ is continuous in $g$ and $\Gamma$-equivariant.  

\subsubsection{Construction of fundamental domain and verification of its properties}
Let $S\subset H$ be a Siegel  fundamental set for $\hat \Gamma $ in $H$ containing $e$; that is (see for instance \cite[VIII.1]{MR1090825})
\begin{enumerate}
\item $\bigcup_{\gamma\in \hat \Gamma} S\gamma = H$;
\item the set $\{ \gamma : S\gamma \cap S\neq \emptyset\}$ is finite;
\item the function $g\mapsto d_H(e,g)$ is $L^p$ on $S$ with respect to the Haar measure for every $1\le p<\infty$. 
\item  
$S$ is left $K$-invariant.
\end{enumerate}
 We say an element of $h\in S$ is \emph{well-positioned in $S$} if, denoting injectivity radius of $H/\hat \Gamma$ at $h\in H$     by  $r^{\hat \Gamma}(h) $,   we have that $V_{r^{\hat \Gamma}(h)}( h) \subset S$.
 Enlarging the cusp parameters of $S$, the set $S$ can be chosen so that 
  \begin{enumerate}[resume]
\item   the well-positioned elements of $S$ form a fundamental set for $\hat \Gamma$.
\end{enumerate}
We will moreover assume $e\in S$ is well-positioned.

An additional property of  $S$ that follows from the construction of the Borel-Serre compactifications is that  $\hat K\bs S$ has compact closure    in $\overline X^{BS}$.  It follows that the set $\{  \langle \cdot, \cdot\rangle_{h} : h\in S\}$ is a uniformly comparable  family of metrics. 

Let $ S' \subset S$ denote the set of well-positioned element of $S$.  Write  $\bar S\subset \bar G$ for  $\bar S = \Phi(  S')$.  Then $\bar S $ is    fundamental set for $\bar \Gamma$ in $\bar G$.  Moreover, as $S$ is $\hat K$-saturated we have that  $\Phi\inv (\bar S) \subset S$.  
It follows that  $\{ \langle \cdot, \cdot\rangle_{\bar g} : \bar g\in \bar S\}$ is a uniformly comparable family of metrics.  Finally let $\td S= \Psi\inv \bar S$.  
Then 
\begin{enumerate}
\item $\td S$ is a fundamental set for $\Gamma$ in $G$;
\item the family $\{ \langle \cdot, \cdot\rangle_{g} : g\in \td S\}$ is a uniformly bounded family of metrics;
\item the family  $\{ \langle \cdot, \cdot\rangle_{g} : g\in G\}$ is $\Gamma$-equivariant and continuous;
\item as all quotients and extensions are by compact kernels, it follows that $g\mapsto d_G(e, g)$ is $L^p$ on $\td S$ with respect to Haar measure;
\item for $g\in \td S$, the image $\Psi(V_{r^\Gamma(g)}(g))$ is contained in $\Phi(S)$
\end{enumerate}
Let $D\subset \td S$ be an almost-open, Borel, fundamental domain for $\Gamma$ in $G$ containing $e$.  Then the desired properties of the family $ \langle \cdot, \cdot\rangle_{g} $ and the fundamental domain $D$ hold.  

\subsubsection{Induced distance on $M^\alpha$}\label{sec:metric}
Using the $\Gamma$-equivariant family of metrics $\{\langle \cdot,\cdot \rangle_g: g\in G\}$ constructed above  and using the right invariant metric on $G$, we endow $G\times M$ with a continuous Riemannian metric such that $\Gamma$ acts by isometries.  This   induces a Riemannian metric on $TM^\alpha$ and corresponding distance function $d(\cdot,\cdot)$ on $M^\alpha$.


\subsection{Some  estimates}
Equip $M$ with any $C^\infty$ Riemannian metric; by compactness, all estimates are independent of the choice of metric.  
Let $\exp_x\colon T_x M\to M$ be the Riemannian exponential map at $x$ and fix  $r_0\le 1$ to be smaller than the injectivity radius of $M$.


Write $B_x(r) \subset T_xM$ for the norm ball $B_x(r) = \{ v\in T_xM : \| v\| < r\}$.
Given a diffeomorphism $f\colon M\to M$ let $$\td f_x \colon U_{x,f}\subset  B_x(r_0)\subset T_xM \to B_{f(x)}(r_0)\subset T_{f(x)}M$$ be the diffeomorphism defined by $$\td f_x:= \exp_{f(x)} \inv \circ f \circ \exp_x $$  on the maximal domain $U_{x,f}$ on which it is defined.  
Given $U\subset U_{x,f}$ 
define the local $C^1$ and  H\"older norms of $\restrict{\td f_x}U\colon U\to B_x(r_0) \subset T_{f(x)}M $ to be 
   $$\|D \td f_x \|_U = \sup_{ v\in U } \|D_v\td f_x\|, \quad \quad   \Hol_U (D\td f_x):= \sup_{v\neq w\in U} \dfrac{\|D_v\td f_x-  D_w\td f_x\|}{\|v-w\|^\beta}.$$

If $f\colon M\to M$ is $C^{1+\beta}$, define
\begin{enumerate}
\item $\|Df\|  := \sup_{x\in M} \|D \td f_x \|_{U_{x,f}}$, 
and 
\item $\hol (Df) := \sup_{x\in M} \Hol_{U_{x,f}} (D\td f_x)$.
\end{enumerate}
Compactness of $M$ ensures $\|Df\|$ and $\hol (Df)$ are finite.  

We have the following elementary estimate. 
\begin{claim} 
Let $f,g\in \Diff^{1+\beta}(M)$.  Given $x\in M$ and $U\subset U_{x,g}\subset T_xM$ such that $$\td g_x(U)\subset U_{g(x),f}$$ 
we have 
$$\Hol _U(D\wtd{(f\circ g)}_x) \le \|Df\|   \Hol ( D  g ) +  \|Dg\|  ^{1+\beta} \Hol ( D  f ).$$
\end{claim}\begin{proof}

For $v,u\in U$ and $\xi$ with $\|\xi\|=1$  
\begin{align*}\| D_v & \wtd{(f\circ g)}_x)\xi -  D_u \wtd{(f\circ g)}_x)\xi  \|  = 
\|D_{\td g_x (v) } \td f_{g(x)} D_v \td g_x  \xi - D_{\td g_x (u) }  \td f_{g(x)} D_u \td g_x  \xi \| \\&
\le \|D_{\td g_x (v) } \td f_{g(x)} D_v \td g_x  \xi  -D_{\td g_x (v) }  \td f_{g(x)} D_u \td g_x  \xi \| 
\\&
\quad \quad +\|D_{\td g_x (v) }  \td f_{g(x)} D_u \td g_x  \xi  - D_{\td g_x (u) }  \td f_{g(x)} D_u \td g_x  \xi \| 
\\&
\le \|D\td f_{g(x)}\| 
\|D_v \td g_x     - D_u \td g_x   \| 
+ 
 \|D\td g_x\| 
 \|D_{\td g_x (v) }  \td f_{g(x)}    - D_{\td g_x (u) }  \td f_{g(x)}   \| 
   \\&
\le \|Df\| \Hol_U ( D \td  g_x )d(u,v)^\beta+ \|Dg\|   \Hol _{\td g_x (U)} ( D \td f_{g(x)}  ) d\left (\td g_x (v),\td g_x (u)\right)^\beta \\&
\le \|Df\| \Hol_U ( D \td  g_x )d(u,v)^\beta+ \|Dg\| ^{1+\beta}   \Hol _{\td g_x (U)} ( D \td f_{g(x)}  ) d(u,v)^\beta.   \qedhere
\end{align*}

\end{proof}

In particular, we have the following.  
\begin{claim}\label{claim:2.5}
Let $g_i\in \Diff^{1+\beta}(M)$, $i= \{1, 2, \dots, \ell\}$ and fix $C$ with $\|  D g_i\| \le C $ and $\Hol ( D  g_i)\le C $.  Given $n\ge 0$ and    $$U\subset B_x (C^{-n} r_0) \subset T_xM$$ with $h = g_{i_1} \circ \dots\circ g_{i_n}$ 
we have 
 \begin{enumerate}
\item $\|D \td h_x\|_U \le C^n$ and 
\item $\Hol _U(D\td{h}_x)\le n C^{n(1+\beta)}$ for every $x$.  
\end{enumerate}
\end{claim}

\subsection{Construction of dynamical charts}\label{sec:charts}

 Let $D\subset G$ be the almost open,  fundamental domain for $\Gamma$ constructed in Section \ref{sec:FD}.   
In the sequel, we often  use the measurable parameterization  $ D\times M$  of $M^\alpha = (G\times M)/\Gamma$.

Fix a globally defined, Borel family of isometric identifications $\tau_x\colon T_xM \to \R^n$.   
With respect to any fixed background $C^\infty$ Riemannian metric on $M$, let  $\exp_x\colon T_xM\to M$ denotes the Riemannian exponential map at $x$ and  $r_0$ denote the injectivity radius of $M$. 
	Let $\R^k = \lieg\oplus \R^n$  be equipped with the product Euclidean metric where  $k= \dim G + \dim M$.  

 Given $p=(g,x)\in D\times M$ let $\rho(g) = \frac 1 2  \min\{r^\Gamma(g), r_0\}$ and let 
$$\phi_{p} \colon \R^k (\rho (g))   \to M^\alpha$$ be the natural embedding
	$$\phi_p \colon (X, v) \mapsto \left(\exp(X) g, \exp_x(\tau_x\inv v)\right)/\Gamma$$
where   we write  $\R^k(r):= \{ v\in \R^k : \|v\| <r\}.$
 We immediately verify that, relative to the induced   metric in \ref{sec:metric}, the charts $\phi_{(g,x)}$ are $C^1$ with $\|D\phi_{(g,x)}\|$   uniformly bounded; in particular relative to the distance function $d$ in  \ref{sec:metric} the charts are uniformly  bi-Lipschitz.  
As the injectivity radius $r^\Gamma(g\Gamma)$ is comparable to the distance from $g\Gamma$ to a fixed base point $x_0\in G/\Gamma$ we have
that $g\mapsto -\log(\rho(g))$ is $L^q$ with respect to the Haar measure for all $1\le q<\infty$.

Recall we let $A$ be the analytic subgroup of $G$ corresponding to $\liea$.  Fixing a basis for $\liea$, via the parameterization   $\exp\colon \liea\to A$ we identify $A$ with $\R^d$ where $d\ge 2$ is the rank of $G$.  Below, we consider an arbitrary lattice  $\Z^d\subset A$  and fix a finite, symmetric, generating set $F= \{ s_j : 1\le j\le m\}$ for $\Z^d$.

Following the notation of \cite{AWB-GLY-P1}, we let $U= U_0= \Lambda=  D\times M= M^\alpha$ for any such $\Z^d$ and $F$.    

In the sequel, we will be concerned with $A$-invariant measures $\mu$ on $M^\alpha$ that project to the Haar measure on $G/\Gamma$. 
Note  in the case that $G$ has compact factors,  the Haar measure on $G/\Gamma$ need not be $A$-ergodic.  However,   from the pointwise ergodic theorem we have the following.
\begin{claim}\label{claim:ergdecom}
Almost every $A$-ergodic component of the Haar measure on $G/\Gamma$ is $G'$-invariant.  
\end{claim} 
Indeed, almost every $A$-ergodic component contains   $G^\xi$ for every nonzero coarse root $\xi\in \hat \Sigma$ and the $G^\xi$ generate all of $G'$.  
Similarly, every $G'$-invariant measure on $G/\Gamma$ is an $A$-ergodic component of the Haar measure on $G/\Gamma$.  

\begin{proposition}\label{charts ok}
Let $\mu$ be an $A$-invariant probability  measure  on $M^\alpha$ projecting to a $G'$-invariant measure on $G/\Gamma$.  Then for any lattice  $\Z^d\subset A\simeq \R^d$ and any finite, symmetric, generating set $F= \{ s_j : 1\le j\le \ell\}$ for $\Z^d$ the standing hypothesis of \cite[Section 3.1]{AWB-GLY-P1} hold relative to the charts $\{\phi_p:p\in M^\alpha\}$ above. 

That is, there are 
measurable functions $r\colon D \to (0, 1]$ and $ C \colon D \to [1,\infty)$ and a constant $L$ 
with 
$$\text{$r(g) \le \rho(g)$, $-\log r(g) \in L^q(D)$, and  $\log C(g) \in L^q(D)$ for all $1\le q <\infty$}$$  
such that, writing   $$r(p) = r(g), \rho(p) = \rho(g), C(p) = C(g)$$ 
for $p=(g,x) \in M^\alpha = D\times M$, we have
\begin{hyp}
\item \label{H1} $\phi_p\colon \R^k(\rho(p) )\to M^\alpha$ is a $C^1$ diffeomorphism onto its image with $\phi_p(0) = p$;
\item    \label{H2} 
$ \|D\phi_p  \|\le L$ and  $ \|D\phi_p \inv \|\le L$;  
in particular 
$\phi_p\colon \R^k(\rho(p) )\to (U, d)$ is a bi-Lipschitz embedding with $L\inv \le \lip(\phi_p)\le L$.
\end{hyp}
Moreover,  for each $m\in F$, setting  $f(\cdot)= \td \alpha(m, \cdot)$ we have for $ p\in  M^\alpha $ that 
\begin{hyp}[resume]
\item  \label{H3} the map \begin{equation} \label{eq: f hat} \td f_p:= \phi_{f(p)}\inv  \circ f \circ \phi_p \end{equation}
is well defined on $\R^k(r(p))$ with range contained in $\R^k(\rho(f(p)))$;
\item  \label{H4} $\td  f_p\colon \R^k( r(p)) \to \R^k(\rho(f(p))) $ is uniformly $C^{1+\beta}$
with $$\|\td f_p\|_{1+\beta}\le  C(p)$$

\end{hyp}

\begin{hyp}[resume]
\item  \label{H5}\label{H6}  \label{hyp:int} for every $n\in \Z^d$, $\displaystyle \left(p\mapsto  \log^+\|D_p\td \alpha (n)\| \right) \in L^q(\mu)$ for any $1\le q<\infty$; in particular $\displaystyle \left(p\mapsto  \log^+\|D_p\td \alpha (n)\| \right) \in L^{d,1}(\mu).$
\end{hyp}
\end{proposition}
Here $L^{d,1}(\mu)$ is the Lorentz integrability space (see \cite{MR0033449}.)  We have $L^{p}(\mu)\subset L^{d,1}(\mu) $ for any $p> d$.   The assertion that $\displaystyle \left(p\mapsto  \log^+\|D_p\td \alpha (n)\| \right) \in L^{d,1}(\mu)$ guarantees the cocycle satisfies the hypotheses of the higher-rank multiplicative ergodic theorem.  
 As $-\log\rho, -\log r$ and $\log C$ are $L^d$ on the domain $D$, it follows that, in the terminology of \cite{AWB-GLY-P1}, they are slowly growing functions over the action of $\Z^d$.  
\begin{proof}
Fix a finite, symmetric     generating set $S= \{ \gamma_i: 1\le i\le \ell\}$ for $\Gamma$.  For each $1\le i\le \ell$ take $g_i = \alpha (\gamma_i) \colon M\to M$ and set $\hat C>1$ with 
\begin{enumerate}
	\item $\|D g_i \| \le \hat C$, and 
	\item $\hol (Dg_i) \le \hat C$.  
\end{enumerate}
Let $S = \{\gamma_i\}$ be a fixed finite generating set for the  lattice $\Gamma \subset G$.  
Let $d_{\text{word}}$ denote the corresponding word metric on $\Gamma$.  
Let $d_G$ denote the distance on $G$ induced by the right-invariant metric on $G$.  Note that $d_G$ restricts to a metric on $\Gamma \subset G$.  
It follows from \cite{MR1828742} that if $\Gamma$ is a higher-rank lattice as introduced in Section \ref{sec:1},   the metrics $d_{\text{word}}$ and $d_{G}$ are  {quasi-isometrically equivalent}: there are $A>1$ and $B>0$ such that for all $\gamma, \hat \gamma\in \Gamma$ we have 
$$ A\inv d_G(\gamma, \hat \gamma) - B \le d_{\text{word}}(\gamma, \hat \gamma) \le A d_G(\gamma, \hat \gamma) + B.$$

Now consider any lattice $\Z^d$ in $A\simeq \R^d$ and finite symmetric generating set $F$.    
Given $g\in D$ and $s_j\in F$ let  $\gamma_j(g)$ be such that $s_j g \in D\gamma_j(g)$.  Define $$N(g)= \max_{s_j\in F} \{ d_{\text {word}} (e, \gamma_j(g))\}.$$
We have  
\begin{align*}
d_{\text {word}} (e, \gamma_j(g)) 
	&\le A\left[d(e,g) + d(e, s_j)+ d(e, s_j g ( \gamma_j(g))\inv )\right] + B.
\end{align*}

Let $\nu$ denote the image of $\mu$ in $G/\Gamma$ and naturally consider $\nu$ as a measure on $D$.  Recall that $C$ denotes the maximal compact normal subgroup of $G$.  We have that $C$ and $G'$ commute whence $C$ acts transitively on the set of  $G'$-ergodic components of the Haar measure.  
As $g\mapsto d(e,g)$ is in $L^q(D, \text{Haar})$ for all $1\le q<\infty$ and as  $C$ has bounded diameter and acts transitively on $G'$-ergodic components,  we have that $g\mapsto d(e,g)$  is in $L^q(D, \nu)$ for all $1\le q<\infty$.  
Also, as the map $D\to D$ given by $g\mapsto s_j g ( \gamma_j(g))\inv$ preserves the Haar measure,  
it follows that  $g\mapsto N(g)$ is in $L^q(D,\nu)$ for all $1\le q<\infty$.

We set $r(g,x) = r(g):= \hat C^{-N(g)} \rho(g)$.  
We have  that $0<r(g,x)\le \rho(g)$ for every $(g,x)\in D\times M$.   Moreover,
we have  that $$\int (-\log (r(g,x)) )^q \ d  \mu(g,x)= \int _D( -\log (r(g)))^q \ d  \nu(g) < \infty.$$

Given $s_j\in F$, let $f= \td \alpha (s_j)$.   
Write $\td f_{(g,x)}\colon \R^k(r(g))\to \R^k(\rho(f(g))$ for 
$$\td f_{(g,x)} :=  \phi_{f(g,x)}\inv \circ f\circ  \phi_{(g,x)}.$$ \ref{H3} then follows. 
From Claim \ref{claim:2.5} we have 
	$$\|D\td f_{(g,x)}\| \le \hat C^{N(g)}, \quad 
	\Hol (D\td f_{(g,x)}) \le N(g) \hat C^{N(g)(1+\beta)}$$
whence  \ref{H4} follows.  
Moreover, we have that the function 
\begin{equation} \label{eq:cocycle} (g,x)\mapsto \log \| D_0 \td f_{(g,x)}\|\end{equation} is $L^q(\mu)$ for every $1\le q<\infty $.  From the cocycle property, \ref{H5} follows for all elements of the action.  
\end{proof}

\section{Lyapunov exponents, coarse foliations, and conditional entropy} 
For this section we consider the restriction of the action $\td \alpha$ on $M^\alpha$ to $A$.  Take  $\mu$ to be an $A$-invariant probability measure on $M^\alpha$.  Let $\nu= \pi_*(\mu)$ be the projection of $\mu$ to $G/\Gamma$.   In the case that $\Gamma$ is not cocompact, assume the projection $\nu$ is $G'$-invariant  so that the charts in Section \ref{sec:charts} satisfy properties \ref{H1}--\ref{H6} of Proposition \ref{charts ok}.   

\subsection{Lyapunov exponent functionals}
From the $L^{d,1}$ integrability of \ref{H6} of Proposition \ref{charts ok}
it follows that the restriction to  $A$ of the derivative cocycle $D\td \alpha$ on $(M^\alpha, \mu)$  
satisfies the hypotheses of the Oseledec's multiplicative ergodic theorem in every direction $s\in \R^d$ (see \eqref{eq:rays} below.)
Moreover, we have uniform convergence along spheres  guaranteed by the stronger conclusions of the higher-rank Oseledec's multiplicative ergodic theorem.

Equip $A\simeq \R^d$ with any norm $|\cdot|$.  

\begin{theorem}[{Higher-rank multiplicative ergodic theorem; \cite[Theorem 2.4]{AWB-GLY-P1}}]\label{thm:OMT}
Let $\mu$ be any $A$-invariant measure on $M^\alpha$ satisfying \ref{H6} of Proposition \ref{charts ok}.  
Then there exist \label{thm:higherrankMET}
	\begin{enumerate}
	\item a full measure, $A$-invariant subset $\Lambda_0\subset M^\alpha$;
	\item an $A$-invariant measurable function $r\colon \Lambda_0\to \N$;
	  \item an $A$-invariant   measurable family of  linear functionals $\lambda_i(p)\colon A\to \R$ for $1\le i\le r(p)$;
	\item and a family of mutually transverse, $\restrict{D\td\alpha}{A}$-invariant,  measurable subbundles $E_{\lambda_i}\subset TM^\alpha$ with $T_{p}M^\alpha = \bigoplus_{i=1}^{r(p)} E_{\lambda_i}(p)$ for $p\in \Lambda_0$ 
	\end{enumerate}
	such that 
	\begin{equation}\label{eq:MET}\lim_{s\to \infty} \frac {\log \| D_p\td \alpha(s) (v)\| - \lambda_i(p) (s)}{|s|}\end{equation}
	for all $ v\in  E_{\lambda_i}(p)\sm \{0\}$.
\end{theorem}

	From \eqref{eq:MET}, for almost every $p\in M^\alpha$ and  every $s\in A$ we have convergence along rays
	\begin{equation}\label{eq:rays} \lim_{k\to \infty} \frac 1 k \log \| D_p\td \alpha(k s) (v)\| = \lambda_i(p) (s)\end{equation}
	for all $ v\in  E_{\lambda_i}(p)\sm \{0\}$.  
	The linear functionals $\lambda_i(p)\colon A \to \R$ are the \emph{Lyapunov exponent functionals}.  The dimension of the corresponding   $E_{\lambda_i}(p)$ is the \emph{multiplicity of $\lambda_i(p)$}.   

Recall the two $D\td \alpha$-invariant subbundles $E^F$ and $E^G$ of $TM^\alpha$.  We may restrict the derivative cocycle $\{D\td\alpha(s): s\in A\}$ to either of the two $A$-invariant distributions $E^F$ or $E^G$. These restrictions satisfy the hypotheses of the higher-rank multiplicative ergodic theorem.  For the restricted cocycles, we obtain Lyapunov exponent functionals $\{\lambda^F_i(p)\}$ and $\{\lambda^G_j(p)\}$ and splittings $E^F (p) = \oplus E^F_{\lambda^F_i(p)}(p), 1\le i\le r^F(p)$ and $E^G (p) = \oplus E^G_{\lambda^G_j(p)}(p), $ for $1\le j\le r^G(p)$  defined on a full measure $A$-invariant subsets.  
By a direct computation, we have that  the linear functionals $\{\lambda^G_j(p)\}$ coincide  with $\Sigma$, the restricted roots of $\lieg$ with respect to $\liea$. 
  In particular, the number $r^G(p)$, the functions $\{\lambda^G_j(p)\}$, and the subspaces $E^G _{\lambda^G_j(p)}(p)$ are defined at every point $p\in M^\alpha$ and are independent of $p$.    

Below, we write $\calL(p)$, $\calL^F(p)$ and $\calL^G(p) = \Sigma$, respectively, for the corresponding collections of Lyapunov exponent functionals at the point $p$ for the derivative cocycle and its restrictions to $E^F$ and $E^g$.   If $\mu$ is $A$-ergodic we write  $\calL(\mu)$, $\calL^F(\mu)$ and $\calL^G(\mu) $ or simply $\calL$, $\calL^F $ and $\calL^G  $ if the measure is understood.  

\subsection{Coarse Lyapunov exponents and coarse Lyapunov manifolds}
For this section assume that  $\mu$ is $A$-ergodic and that the charts in Section \ref{sec:charts} satisfy properties \ref{H1}--\ref{H6} of Proposition \ref{charts ok}.  Note that Lyapunov exponents and dimension of the corresponding subspaces are independent of the point \as

As with restricted roots, we group Lyapunov exponent functionals into \emph{coarse equivalence classes} by declaring that two exponents are equivalent if they are positively proportional.  We write  $\hat \calL$ for the   equivalence classes of coarse exponents.  For $\chi\in \hat \calL$ we write $E_\chi(p) = \bigoplus_{\lambda \in \chi} E_\lambda(p)$.  

Recall that we equipped $TM^\alpha$ with a Riemannian metric which, in turn, induces a distance $d$ on $M^\alpha$.  
Given $s\in A$ and $p\in M^\alpha$  we write $$W^u_s(p): = \left\{y\in M^\alpha : \limsup_{n\to -\infty } \frac{1}{n} \log d\left(\td \alpha (ns) (p), \td \alpha (ns) (y)\right) <0\right\} $$ for the unstable  manifold through $p$ for the action of $s\in A$ on $M^\alpha$.  
For $\mu$-almost every $p\in M^\alpha$, we have that $W^u_s(p)$ is a connected,  injectively immersed, $C^{1+\beta}$ manifold with  $T_p W^u_s(p)= \bigoplus_{\lambda \in \calL: \lambda(s)>0} E_\lambda(p)$.  
Observe that given $s\in A$, the global unstable manifolds $\{W^u_s(p):p\in M^\alpha\}$ form  a (generally non-measurable) partition of $(M^\alpha, \mu)$.  

Let $\Z^d$ be any lattice in $A\simeq \R^d$. Given a coarse Lyapunov exponent $\chi\in \hat \calL$ we write 
$W^\chi(p)$ for the path connected (relative to the immersed topologies) component of $$\bigcap _{\{s\in \Z^d:\chi(s)>0\} }W^u_s(p)$$ containing $p$.  
$W^\chi(p)$ is   the \emph{coarse Lyapunov manifold} corresponding to $\chi$ through $p$.  
For \ae $p$, $W^\chi(p)$ is a $C^{1+\beta}$ injectively immersed manifold with  $T_pW^\chi(p) = E_\chi(p)$ (see  \cite{AWB-GLY-P1}).  
We let  $\scrW^\chi$ denote the partition of $(M^\alpha, \mu)$ into coarse Lyapunov manifolds $W^\chi(p)$.  
 In the terminology of \cite{AWB-GLY-P1}, $\scrW^\chi$ is a $C^{1+\beta}$-tame, $\restrict{\td \alpha}{A}$-invariant, measurable foliation.  
Note that the partition $\scrW^\chi$ is defined independently of the choice of lattice $\Z^d\subset A$ in that for any two choices of lattice, the corresponding partitions coincide modulo $ \mu$.  

Similarly,  in the terminology of \cite{AWB-GLY-P1}, the partition $\Gol$ of $M^\alpha$ into $G$-orbits and the partition $\Fol$ of $M^\alpha$ into fibers form  $C^{1+\beta}$-tame, $\td \alpha$-invariant, measurable foliations. 
We similarly define $W^{\chi^F}(p)$ and $W^\xi(p)$ for the coarse Lyapunov manifolds   associated to coarse 
fiberwise Lyapunov exponents $\chi^F\in \hat \calL^F$ and coarse roots $\xi\in \hat \Sigma$.  Note that if $\xi\in \hat \Sigma$ then $W^\xi(p)$ is simply the orbit $ \td\alpha (G^\xi)( p)$ of $p$ by the unipotent subgroup $G^\xi = \exp \lieg^\xi$ of $G$.  We similarly define measurable foliations $\scrW^{\chi^F}$ and $\scrW^\xi$ given by the partitions into fiberwise coarse manifolds and orbits of coarse root groups. 

\subsection{Conditional entropy, entropy product structure, and  coarse-Lyapunov Abramov--Rohlin formula}
Recall  for $s\in A$  the $\mu$-metric entropy of $\td \alpha(s)$ is  $$ h_\mu(\td  \alpha(s) ):= \sup\{ h_\mu( \td \alpha(s), \eta ) \}$$
where the supremum is over all measurable partitions $\eta$ of $(M, \mu)$ and $h_\mu( \td \alpha(s), \eta ) $ is given by the mean conditional entropy $$h_\mu( \td \alpha(s), \eta ) = H_\mu(\eta^+ \mid \td \alpha(s) \eta^+)$$
where $\eta^+= \bigvee_{i=0}^\infty   \td \alpha(s^i)  \eta$.

Given the partition $\scrW^\chi$ into   coarse Lyapunov manifolds for $\chi\in \calL$,  for $s\in A$ with $\chi(s)>0$ we   define  the \emph{conditional metric entropy} of $\td \alpha(s)$ \emph{relative to $\scrW^\chi$} as follows:  A measurable partition $\xi$ of $(M^\alpha, \mu)$ is said to be \emph{subordinate to $\scrW^\chi$} if, for \ae $p$, 
\begin{enumerate}
\item the atom  $\xi(p)$ is contained in $ \scrW^\chi(p)$,  
\item the atom  $\xi(p)$ contains a neighborhood of $p$ in $\scrW^\chi(p)$, and 
\item the atom  $\xi(p)$ is precompact in $ \scrW^\chi(p)$.  
\end{enumerate}
The  {conditional metric entropy} of $\td \alpha(s)$  {relative to $\scrW^\chi$}
is $$ h_\mu(\td \alpha(s) \mid \scrW^\chi):= \sup\{ h_\mu(\td \alpha(s), \eta\vee \xi) \}$$
where the supremum is over all partitions $\xi$ subordinate to $\scrW^\chi$ and all measurable partitions $\eta$. 

From  \cite{AWB-GLY-P3} we have the following result which states that entropy behaves like a product along coarse Lyapunov manifolds.  
\begin{proposition}[{\cite[Corollary  13.2]{AWB-GLY-P3}}]\label{en:PS}
For $s\in A$ 
$$h_\mu( \td \alpha(s) )= \sum_{\chi(s)>0} h_\mu(\td \alpha(s) \mid \scrW^\chi).$$
\end{proposition}

Given a coarse exponent  $\chi\in \hat \calL$ we write  $\chi(F)\in \hat\calL^F$ for  the unique fiberwise coarse exponent with $$\chi(F) = c \chi$$ for some $c>0$ if such $c$ exists and $0$ otherwise.  Similarly, define $\chi(G)$ to be the unique coarse restricted root $\hat \xi\in \hat \Sigma$ that is positively proportional to $\chi$ and $0$ otherwise.  Note that given a non-zero coarse Lyapunov exponent  $\chi\in \hat \calL$, at least one of $\chi(F)$ or $\chi(G)$ is non-zero.   

Let $\nu$ denote the image of $\mu$ under $\pi\colon M^\alpha \to G/\Gamma$.  
From the Abramov-Rohlin formula (c.f. \ \cite{MR0476995,MR1179170}), we may decompose entropy of $\mu$ into the sum of the  entropy along  fibers and the entropy of the factor: for any $s\in A$
\begin{equation}\label{eq:AR}
h_\mu(\td \alpha(s)) = h_{\nu} (s) + h_\mu(\td \alpha(s) \mid \Fol).
\end{equation}
Here $\Fol$ is the partition into preimages of the projection $\pi$ and $h_{\nu} (s) $ is the metric entropy of the translation by $s$ on $(G/\Gamma,\nu)$.  
From  \cite{AWB-GLY-P3}, we have a similar decomposition   into fiber and factor entropy along coarse manifolds. 
\begin{proposition}[{\cite[Theorem 13.4]{AWB-GLY-P3}}]\label{prop:AR}
Let $s\in A$ be such that $\chi(s)>0$.  Then \begin{equation}\label{eq:ineq} h_\mu(\td \alpha(s)\mid \scrW^{\chi}) = h_{\nu}( s\mid \chi(G))+ h_\mu (\td \alpha(s)\mid \scrW^{\chi(F)}).\end{equation}
\end{proposition} 
Above, $h_{\nu}( s\mid \chi(G))$ denotes the metric entropy of translation by $s$ on $(G/\Gamma,\nu)$ conditioned on the partition of $(G/\Gamma,\nu)$ into orbits of $G^{\chi(G)}$.  
Note that for our applications below,  if $\chi(F)=0$ then $h_\mu (\td \alpha(s)\mid \scrW^{\chi(F)})=0.$

 Proposition \ref{prop:AR} is a special case of \cite[Theorem 13.4]{AWB-GLY-P3} which establishes an Abramov-Rohlin formula for entropy subordinated to  coarse Lyapunov manifolds for  two smooth $\Z^d$-actions, one of which is a measurable factor of the other.  In the current setting, our  factor map $\pi\colon M^\alpha \to G/\Gamma$ is smooth and we obtain  Proposition \ref{prop:AR} directly from Proposition \ref{en:PS}.  We include a proof of  Proposition \ref{prop:AR} in our current setting.  
 \begin{proof}[Proof of  Proposition \ref{prop:AR}]  
Note that, as the map $\pi\colon M^\alpha\to G/\Gamma$ is smooth, every coarse restricted root $\hat \xi\in \hat \Sigma$ for the action of $A$ on $G/\Gamma$  coincides with some coarse Lyapunov exponent  $\chi\in \hat \calL$ for the action of $A$ on $(M^\alpha, \mu)$; in particular, every  $\hat \xi\in \hat \Sigma$ is of the form $\hat \xi =  \chi(G)$ for some $ \chi\in \hat \calL$.  

Given $\chi\in \hat \calL$,   set $\bar \chi = \chi(G)$ and take $s\in A$ with $\chi(s)>0$.  If $\bar \chi = 0$ take  $\bar \eta $  to be the point partition on $G/\Gamma$.  Otherwise, take $\bar \chi$ to be a measurable partition of $G/\Gamma$ such that
\begin{enumerate}
	\item $s\inv \cdot \bar \eta \ge \bar \eta$;
	\item the atom $\bar \eta(x)$ of $\bar \eta$ containing $x$ is contained in the $G^{\bar \chi}$-orbit of $x$ and  contains an open neighborhood of $x$ in the $G^{\bar \chi}$-orbit;
	\item $\bigvee_{n\in \N} s^{-n} \cdot \bar \eta$ is the point partition.  
\end{enumerate}
Let $\eta = \pi\inv (\bar \eta)$.   Take  $\zeta$ to be a measurable partition of $M^\alpha$ such that 
\begin{enumerate}
	\item  $\td \alpha(s\inv)(\zeta)\ge \zeta$;  
	\item the atom $ \zeta(x)$ of $  \zeta$ containing $x$ is contained in   $\scrW^\chi(x)$ and  contains an open neighborhood of $x$ in  $\scrW^\chi(x)$ for almost every $x$;
	\item $\bigvee_{n\in \N} \td \alpha (s^{-n}) ( \zeta)$ is the point partition. 
\end{enumerate}

The partitions $\bar \eta$ and $\zeta$ satisfy
$$ \text{$ h_\nu(s, \bar \eta) = h_\nu(s\mid \bar \chi)$,    
and $h_\mu(\td \alpha(s),  \zeta\vee \Fol)= h_\mu(\td \alpha(s)\mid \scrW^{\chi(F)}).$}$$
We have the following  standard computation (c.f.\ \cite[Lemma 6.1]{MR2729332}):
\begin{align*}h_\mu(\td \alpha(s) \mid \scrW^\chi) &:= 
h_\mu(\td \alpha(s), \eta \vee \zeta) \\& \le  h_\mu(\td \alpha(s),   \eta) + h_\mu\left (\td \alpha(s),  \zeta\vee \bigvee_{n\in \Z}\alpha(s^n)(\eta)\right) \\
&= h_\nu(s, \bar \eta) + h_\mu(\td \alpha(s),  \zeta\vee \Fol) \\
& = h_\nu(s\mid \bar \chi) + h_\mu(\td \alpha(s)\mid \scrW^{\chi(F)}).\end{align*}
 
Now, fix $\chi_0\in \hat \calL.$  Given any $s\in A$ with $\chi_0(s)>0$ we have from   \eqref{eq:AR} and the analogue of  Proposition \ref{en:PS} applied to the total, fiber, and base entropies (see full formulation in  \cite[Theorem 13.4]{AWB-GLY-P3})  that 
\begin{align*}h_\mu(\td \alpha(s) ) &=
\sum_{\chi(s)>0}
h_\mu(\td \alpha(s) \mid \scrW^\chi)  \\
&\le \sum_{\chi(s)>0} h_\nu(s\mid  \chi(G)) +  \sum_{\chi(s)>0}  h_\mu(\td \alpha(s)\mid \scrW^{\chi(F)})\\
& =  h_\nu(s) + h_\mu(\td \alpha(s)\mid \Fol)\\
&= h_\mu(\td \alpha(s)).
\end{align*}
Since entropies are non-negative quantities, it follows that  $$ h_\mu(\td \alpha(s) \mid \scrW^\chi) =h_\nu(s\mid  \chi(G)) + h_\mu(\td \alpha(s)\mid \scrW^{\chi(F)}) $$
 for all $\chi\in \hat \calL$ with $\chi(s)>0$.  
 \end{proof}

 \section{Conditional measures and criteria for invariance}
 Let $G$ be as in the introduction.  That is $G= G_1\times \dots\times G_\ell$ is the direct product of almost-simple Lie groups.   Let $G'\subset G$  be the product of all non-compact almost-simple factors and $C\subset G$ the product of all compact almost-simple factors.    
  Consider $X$ any locally compact, second countable metric space and suppose that $X$ admits a continuous left $G$-action $x\mapsto g\cdot x$.  
We moreover assume the action is \emph{locally free}; that is, for every $x\in X$ there is a neighborhood $e\in U_x \subset G$ such that the map $U_x\to X, g\mapsto g\cdot x$ is injective.  It follows that for every $x$ we have a canonical  identification of $G$ with a covering space of the orbit $G\cdot x$ given by $g\mapsto g\cdot x$.

\subsection{Conditional measures along orbits of subgroups}
Consider $\mu$  any Borel probability measure  on $X$.  Let $V\subset G$ be a connected Lie subgroup and let $\eta$ be a measurable partition of $(X, \mu)$ such that for almost every $x\in X$, the atom $\eta(x)$ is contained in the $V$-orbit   $V\cdot x$ and contains an open neighborhood of $x$ in the   $V$-orbit $V\cdot x $.  Such a partition is said to be \emph{subordinate to $V$-orbits}.  
As above, we identify   $V\subset G$ with a cover of the $V$-orbit through $x$.  
Taking a family of conditional probability measures for the partition $\eta$ of $(X, \mu)$,  for $\mu$-almost every $x\in X$   we obtain a  probability measure $\mu^\eta_x$ on $G$ whose support contains the identity and is contained in $V$.  
Fix a sequence of measurable partitions $\eta_j$ subordinated to $V$-orbits such that for any compact set $K\subset V$, for almost every $x$ there is a $j$ with   $K\cdot x\subset \eta_j(x)$.  By fixing a choice of  normalization on $V$,  a  standard construction gives  for  almost every $x\in X$   a locally finite  measure $\mu^V_x$, supported on $V$, which is canonical up to the choice of normalization.  To emphasize the lack of uniqueness, we  write $[\mu^V_x]$ for the equivalence class of the measure $\mu^V_x$  up to normalization of the measure.   

The family of measures $\mu^V_x$ have the following characterization: for any partition $\eta$ subordinated to $V$-orbits, there is a function $c^\eta\colon X\to (0,\infty)$  so that if $\{\mu^\eta_x: x\in X\}$ is a family of conditional measures on $(X,\mu)$ associated with the measurable partition $\eta$,  then $$\mu^\eta_x = c^\eta(x)\restrict{ \left(v\mapsto v\cdot x  \right)_*(\mu^V_x)}{\eta(x)}.$$

Note that the subgroups $V$ above need not be unimodular.  
We have the following claim which follows from local disintegration and the definition of the left Haar measure.  
\begin{claim}\label{claim:kklloopp}
Let $V\subset G$ be a connected Lie subgroup.  
Then the measure  $[\mu^V_x]$ coincides with the left Haar measure on $V$ for $\mu$-almost every $x\in M$ if and only if the measure $\mu$ is invariant under the action of $V$.  
\end{claim}
The remainder of this section is devoted to a number of criteria which will guarantee that $[\mu^V_x]$ is the left Haar measure.

\subsection{Invariance from the structure  of parabolic subgroups}
Recall we write $P= MAN$ for the minimal parabolic subgroup of $G$. 
 Let $P'= P\cap G'$.  Then $P'= M'AN$ is the minimal parabolic subgroup of $G'$.
Suppose $\mu$ is a $P'$-invariant, Borel probability  measure on $X$.  
Given a coarse negative root $\xi\in \hat \Sigma_-$ and a nontrivial subgroup $V\subset G^\xi$ such that $\mu$ is $V$-invariant then, as the stabilizer of a measure is a closed subgroup of $G$, it follows from the structure theory of parabolic subgroups (of $G'$)  that $\mu$ is invariant by the full coarse root group $G^\xi$.   
In the case that the subgroup $V$ above varies with the point $x\in X$, we have the following  lemma.  Note that $G^\xi$ is nilpotent so subgroups of $G^\xi$ are unimodular.  
\begin{lemma}\label{prop:aaalll}
Let $\mu$ be a $P'$-invariant measure on $X$ and suppose for some $\xi\in \Sigma_-$, that for $\mu$-a.e.\ $x\in X$ there is a nontrivial, connected Lie subgroup $V(x)\subset G^\xi$ such that $[\mu^{G^\xi}_x]$ coincides with  the Haar measure on $V(x)$. Moreover, assume the assignment $x\mapsto V(x)$ is measurable and  $A$-invariant.  Then the measure $\mu$ is $G^\xi$-invariant.  
\end{lemma}

\begin{proof}
Let  $\{\mu^e_x\}$ denote the  $A$-ergodic decomposition of $\mu$.   
It is enough to verify that the measure $\mu^e_x$ is $G^\xi$-invariant for almost every $x$.  
First note that there is a $s\in A$ so that, for every $x\in X$ and $y \in N\cdot x$, $d(s^k\cdot x, s^k\cdot y)\to 0$.  It follows that the partition into ergodic components is refined by  the measurable hull  of the partition into $N$-orbits.  In particular, for $\mu$-\ae $x$ the measure  $\mu^e_x$ is $N$-invariant. 

Fix a generic $x\in X$.  Let $V$ be the $\mu^e_x$-a.s.\ constant value of $x\mapsto V(x)$.  
Let $H(x)$ be the closed subgroup of $G$ under which $\mu^e_x$ is invariant 
and let $\lieh_x= \Lie(H(x))$.  

As $-\xi$ is a positive coarse restricted root, we have $\lieg^{-\xi}\subset \lieh_x$.  
Moreover, given a non-zero $Y\in \lie(V)$, from the analysis of $\mathfrak{sl}(2, \R)$ triples  in $\lieg$ (see \cite[Lemma 7.73]{MR1920389}), we have that 
$(\mathrm{ad}(Y))^2$ maps $\lieg^{-\xi}$ onto $\lieg^\xi$.  
In particular $\lieg^\xi\subset \lieh_x$ whence  $\mu^e_x$ is $G^\xi$-invariant.  
\end{proof}

\subsection{High-entropy method}
We have the following theorem of Einsiedler and Katok from which we deduce  invariance  along unipotent subgroups      of an $A$-invariant measures based on its support along coarse root spaces.  We say a Lie subalgebra $\lieh\subset \lieg$ is \emph{contracting} if it is invariant under the adjoint action of $A$ and if  there is some $s\in A$ with $$\lieh= \bigoplus_{\xi\in \hat \Sigma: \xi(s)<0}( \lieg^\xi\cap \lieh).$$
Note that any such $\lieh$ is nilpotent, hence unimodular.   
We state a simplified version of the \emph{High Entropy Theorem} from \cite{MR2191228}.

\begin{theorem}[{High Entropy Theorem, \cite[Theorem 8.5]{MR2191228}}]\label{thm:HE}
Let $\mu$ be an $A$-invariant measure on $X$ and let $\lieh\subset \lieg$ be a contracting Lie algebra with corresponding analytic subgroup $H$. Then for $\mu$-\ae $x$ there are Lie subgroups $$H_x\subset S_x\subset H$$ with 
\begin{enumerate}
\item $\mu^H_x$ is supported on $S_x$;
\item $\mu^H_x$ is invariant under left and right multiplication by $H_x$;
\item\label{3} $H_x$ and $S_x$ are connected and their Lie algebras are direct sums of subspaces of root spaces; 
\item \label{4}$H_x$ is normal in $S_x$ and if $\xi,\xi'\in \hat \Sigma$ with $\xi\neq \xi'$ are distinct coarse roots then for $g\in S_x \cap G^\xi$ and $h\in S_x \cap G^{\xi'}$ the cosets $gH_x$ and $hH_x$ commute in $S_x/H_x$;
\item $\mu^{G^\xi}_x$ is left- and right- invariant under multiplication by elements of $H_x\cap G^\xi$.  
\end{enumerate}
\end{theorem}
It follows that the groups $S_x$ and $H_x$ are equivariant under conjugation by $A$; that is $S_{s\cdot x} = sS_xs\inv$.  Unlike in \cite{MR2191228}, we only consider here the adjoint action of $A$ on $\lieg$.  As this action  is semisimple with real roots, it follows that the groups $S_x$ and $H_x$ are normalized by $A$.  In particular, the maps $x\mapsto S_x$ and $s\mapsto H_x$ are constant along $A$-orbits.


\subsection{Invariance from entropy considerations} Consider again $\mu$ an $A$-invariant, $A$-ergodic measure.  
Given a coarse root $\xi\in \hat \Sigma$ let $\scrW^\xi$ be the partition of $X$ into orbits of $G^\xi$.  We have a standard fact (see for example \cite{MR693976}) that if $\mu$ is $G^\xi$-invariant then for $s\in A$ with $\xi(s)>0$, the   entropy of the action of $s$ on $ (X,\mu)$ conditioned along orbits of $G^\xi$ is given by 
$$h_{\mu} (s\mid \scrW^\xi) = \sum_{\beta \in \xi} \beta(s)\dim (\lieg^\beta).$$
The converse also holds. 
\begin{lemma}\label{lem:fullentropyinvariant}
Let $\xi\in \hat \Sigma$ be such that $$h_\mu(s\mid \scrW^\xi) = \sum_{\beta \in \xi} \beta(s)\dim (\lieg^\beta)$$ for some $s\in A$ with $\xi(s)>0$.  
Then $\mu$ is $G^\xi$-invariant. 
\end{lemma}
Indeed, Ledrappier shows in \cite[Theorem 3.4]{MR743818}   that $\mu$ has absolutely continuous conditional measures along $G^\xi$-orbits. Moreover, from the explicit computation of the density function in the proof of \cite[Theorem 3.4]{MR743818} 
 it follows that the conditional measures of $\mu$ along $G^\xi$-orbits coincide with the image of the Haar measure on $G^\xi$.  See also   \cite[(6.1)]{MR819556} for the argument in English.  
 From Claim \ref{claim:kklloopp} it follows that $\mu$ is $G^\xi$-invariant.   

We remark that deriving extra invariance of a measure by verifying that conditional entropy is maximized also underlies the proof of the so-called ``invariance principle'' for fiber-wise conditional measures invariant under a skew product, developed by Ledrappier in \cite{MR850070} and extended in \cite{MR2651382}.

\section{Main   propositions and proofs of Theorems \ref{thm:invmeas}, \ref{thm:measconjfin}, and \ref{thm:measconj}}

\subsection{Non-resonance implies  invariance}
We return to the setting introduced in Section \ref{sec:2}.  
Consider an $A$-invariant, $A$-ergodic measure $\mu$ on $M^\alpha$ satisfying \ref{H6} of Proposition \ref{charts ok}.  We  say a restricted root $\beta\in \Sigma$ of $\lieg$ is \emph{resonant} (with the fiber exponents $\calL^F(\mu)$ of $\mu$) if there exists a $c>0$ and a $\lambda\in \calL^F(\mu)$ with 
	$$\beta = c\lambda.$$
If no such $c$ and $\lambda$ exist, we say $\beta$ is \emph{non-resonant}.  
We similarly say that a fiberwise Lyapunov exponent $\lambda\in \calL^F(\mu)$ is \emph{resonant} (with $\lieg$) if there is a $c>0$ and $\beta\in \Sigma$ with $$\lambda= c \beta.$$
Note that resonance and non-resonance are well-defined on the set of coarse restricted roots $\hat \Sigma $ and coarse fiberwise exponents $\hat \calL^F(\mu)$. 

The proof of Theorem \ref{thm:invmeas} follows directly from the following key proposition.  
\begin{proposition}\label{thm:nonresonantimpliesinvariant}
Let $\mu$ be an $A$-invariant, $A$-ergodic Borel  probability measure on $M^\alpha$ such that the image of $\mu$ in $G/\Gamma$ is $G'$-invariant.  
 Then, given a coarse restricted root $\xi\in \hat \Sigma$   that is non-resonant with the fiberwise Lyapunov exponents of $\mu$, the measure $\mu$ is $G^\xi$-invariant for the action $\td \alpha$.  
\end{proposition}
\begin{proof}
Indeed if $\xi$ is a non-resonant coarse restricted root then $\xi =\chi(G)$ for some coarse exponent $\chi\in \hat \calL$ with $\chi(F) = 0$.  
Let $\nu$ denote the image of  $\mu$ in $G/\Gamma$.  Since $\nu$ is $G'$-invariant, it follows for $s\in A$ with $\xi(s)>0$  that $$h_\nu(s\mid \xi) = \sum_{\beta\in \xi} \beta(s)\dim (\lieg^\beta).$$
From Proposition \ref{prop:AR} and the fact that the partitions  $\scrW^\chi= \scrW^\xi$ coincide in $M^\alpha$, it follows that $h_\mu (\td \alpha(s)\mid \scrW^{\chi(F)})=0$ whence

$$h_\mu(\td \alpha (s) \mid \scrW^\xi) =  \sum_{\beta\in \xi} \beta(s)\dim (\lieg^\beta).$$
The $G^\xi$-invariance of $\mu$ then follows from Lemma  \ref{lem:fullentropyinvariant}.  
\end{proof}

We remark that the proof of Proposition \ref{thm:nonresonantimpliesinvariant} is similar to key steps in  \cite{MR1253197} and \cite{1302.3320} where one deduces extra invariance of a measure by computing conditional entropy, verify the entropy is the maximal value permitted by  the Margulis--Ruelle inequality, and applying Ledrappier's result Lemma  \ref{lem:fullentropyinvariant} to obtain invariance.

\subsection{$P$-invariant measures and the proof of Theorem \ref{thm:invmeas}}
Recall that $P$ is   the minimal standard parabolic subgroup and is hence amenable.  It follows 
that there exists an invariant probability measure $\mu$ for the  restriction to $P$ of the action  $\td \alpha$  on $M^\alpha$. Moreover, it follows that any such measure factors to the Haar measure on $G/\Gamma$.
Fix a $P$-invariant, $P$-ergodic measure $\mu$  on $M^\alpha$.  Recall that $A\subset P$  and the data  $r(\cdot)$, $\lambda_i(\cdot)$, $E_{\lambda_i}(\cdot)$ defined in Theorem \ref{thm:OMT} for the action of $A$ on $(M,\mu)$ as well as the corresponding data $r^F(\cdot)$, $\lambda^F_i(\cdot)$, and $E_{\lambda_i ^F}(\cdot)$ and $r^G(\cdot)$, $\lambda^G_i(\cdot)$, and $E_{\lambda_i ^G}(\cdot)$ for the fiberwise and orbit cocycles.  As observed earlier, the data $r^G(\cdot)$, $\lambda^G_i(\cdot)$, and $E_{\lambda_i ^G}(\cdot)$ are independent of the measure $\mu$ and the point.  We show for $\mu$ as above, the   remaining data is independent of the point.  

\begin{claim}\label{claim:llmm}
Suppose that   $\mu$ is a $P$-invariant, $P$-ergodic measure.  Then the functions $r(\cdot), r^F(\cdot)$, $ \lambda_i(\cdot)$,  $ \lambda^F_i(\cdot)$ and the  dimensions of the corresponding subspaces $E_{\lambda_i}(\cdot)$, $E_{\lambda_i ^F}(\cdot)$ are constant almost surely. 
\end{claim}
 \begin{proof}
Note that $\mu$ is $P$-ergodic  but need not be $A$-ergodic. 
Let $\{\mu^e_p\}_{p\in M^\alpha}$ denote the $A$-ergodic decomposition of $\mu$.
We may select $s\in A$ so that $\beta(s) <0$ for every $\beta \in \Sigma_+$.  By the pointwise ergodic theorem, it follows that ergodic components are refined by the measurable hull of the partition into  $N$-orbits.  Then $\mu^e_p$ is $N$-invariant for almost every $p\in M^\alpha$.  It follows that the data in the claim is constant along $AN$-orbits.

Finally, recall that $P =MAN$ with $M$ 
contained in the centralizer of $A$.  It follows that the data is constant along $M$ orbits.  By the $P$-ergodicity of $\mu$,  the result follows.  
\end{proof}

From  Claim \ref{claim:llmm} it follows that for any $P$-invariant, $P$-ergodic measure $\mu$ on $M^\alpha$, the set of resonant roots depends only on the measure $\mu$ and not the decomposition of $\mu$ into $A$-ergodic components.

Theorem  \ref{thm:invmeas} now follows immediately from Proposition \ref{thm:nonresonantimpliesinvariant}.  
\begin{proof}[Proof of Theorem \ref{thm:invmeas}]
	Let $\mu$ be any $P$-invariant, $P$-ergodic measure on $M^\alpha$. 
	Then $\mu$ is invariant under a closed subgroup $Q\supset G$ with $P\subset Q$. 
	If $\dim M< r(G)$ then there at most $r(G)-1$ fiberwise  Lyapunov  exponent functionals in $\calL^F$, hence at most 
	$r(G)-1$  coarse fiberwise Lyapunov  exponent functionals in $\hat \calL^F$. 
	It follows that there are at most $r(G)-1$ resonant  coarse restricted roots $\xi\in \hat \Sigma$.  From Proposition \ref{thm:nonresonantimpliesinvariant}, it follows that $Q$ is a standard parabolic subgroup with resonant codimension strictly smaller then $r(G)$.  But then $Q = G$ by definition of $r(G)$.  
	
It follows that $\mu$ is a $G$-invariant, Borel probability measure on $M^\alpha$.  From Claim \ref{claim:measureupstairs}, it follows that there exists a $\Gamma$-invariant Borel probability measure on $M$.
\end{proof}

\def\Qin{Q_{\mathrm{In}}}
\def\Qout{Q_{\mathrm{Out}}}

\subsection{Parabolic subgroups associated to conditional measures}  
\label{sec:QinQout}

We continue to assume  $\mu$ is a $P$-invariant, $P$-ergodic measure on $M^\alpha$.   The proof of Theorems \ref{thm:measconjfin} and \ref{thm:measconj} follow from an analysis of the geometry of the measures $[\mu^G_p]$  constructed in the previous section.  

We define subgroups  $\Qin(\mu)\subset \Qout(\mu)$ of $G$  as follows:
Given $p\in M^\alpha$ let
\begin{enumerate}  
\item $\Qin(\mu)$  denote  the largest subgroup of $G$ for which $\mu$ is invariant for the action $\td \alpha$;
\item $\Qout(\mu; p)$ denote the smallest, closed,  $[\mu^G_p]$-co-null subgroup of $G$.  
\end{enumerate}
Note that both $\Qin(\mu)$ and $ \Qout(\mu;p)$ are standard parabolic subgroups.  
As     $P\subset \Qout(\mu; p)$, it follows that $\Qout(\mu; p)$ is constant along $P$-orbits.  By $P$-ergodicity of $\mu$,   we write $\Qout(\mu)$ for the almost-surely constant value of $\Qout(\mu; p)$.  

Theorems \ref{thm:measconjfin} and \ref{thm:measconj} will follow from verifying that  $\Qin(\mu)= \Qout(\mu)$.   We use the criteria in the previous section to  verify this condition.  
First, consider the  case that every fiberwise Lyapunov exponent $\lambda_i^F$  of $\mu$ is resonant with a negative root.  In this setting we immediately obtain that $\Qin(\mu)$ and $ \Qout(\mu)$ coincide.  
\begin{proposition}\label{thm:resonant}
Suppose that for every $\lambda_i^F\in \calL^F$ there is a $\beta\in \Sigma_-$ and $c>0$ with $\lambda_i^F= c\beta$.  Then $\Qin(\mu)= \Qout(\mu)$.
\end{proposition}

We also verify that  $\Qin(\mu)= \Qout(\mu)$ given the combinatorics of the number $m(G)$.

\begin{proposition}\label{thm:High entropy}
Suppose $\lieg$ has no rank-1 simple ideals and that $\Qin(\mu)$ is a maximal parabolic subgroup.  Then $\Qin(\mu)= \Qout(\mu)$.
\end{proposition}

\subsection{Proofs of Theorems \ref{thm:measconjfin} and \ref{thm:measconj}}
Given a $P$-invariant, $P$-ergodic measure $\mu$ as above, let $\td \mu$ denote the locally finite measure on $G\times M$ obtained from lifting $\mu$ on   fundamental domains of $\Gamma$.  Given $g\in G$, let $  \mu_g$ denote the conditional probability measure on $M$ defined by disintegrating $\td \mu$ along fibers and  identifying the fiber $\{g\}\times M$ with $M$.    

As $\td \mu$ lifts $\mu$, we have that $\{\mu_g:g\in G\} $ is $\Gamma$-equivariant:
$$\mu_{g\gamma} = \alpha(\gamma)_*\mu_g.$$  Moreover, as $\mu$ is $\Qin(\mu)$-invariant, 
for almost every $g\in G$, we have that $\mu_{g} = \mu_{qg}$ for every $q\in \Qin(\mu)$.  Let $Q= \Qin(\mu)$.  
We equip $Q\bs G$ with any measure $m$ in the Lebesuge class.  
Let $\bar \mu$ be the measure on $Q\bs G\times M$ given by 
$$ \bar \mu(B) = \int \mu_g(\{x:  (Qg,x)\in B\}) \ d m(Qg)$$
and let 
 $\hat \mu$ be the measure on $ M$ given by $$\hat \mu(B) = \int \mu _g(B) \ d m(Qg).$$
 Note that $\hat \mu$ is image of $\bar \mu$ under the natural projection $\pi\colon Q\bs G\times M \to M$.  

Consider the $\bar \mu$-measurable partition $ \zeta^\pi$ on $Q\bs G \times M$ into level sets of the map $\pi$.  We have that $ \zeta^\pi$ is measurably equivalent to the partition $\{\Qin(\mu)\bs\Qout (\mu)\times \{x\} : x\in M\}$.  In particular, in the case $\Qin(\mu)= \Qout(\mu)$  the following claim follows immediately.  

\begin{claim}\label{prop:msblisom}
If  $\Qin(\mu)= \Qout(\mu)$ then the projection $(Q\bs G\times M, \bar \mu) \to (M, \hat \mu)$ is a measurable isomorphism.  
\end{claim}

Theorems \ref{thm:measconjfin} and \ref{thm:measconj} follow from   $\Gamma$-equivariance of the family $\{\mu_g\}$ and   Claim \ref{prop:msblisom}.

\begin{proof}[Proof of Theorems \ref{thm:measconjfin} and \ref{thm:measconj}]
Let $\mu$ be a $P$-invariant, $P$-ergodic measure on $M^\alpha$.  

First consider the  setting of Theorem \ref{thm:measconjfin} where $\dim(M)= r(G)$.  If there exists a non-resonant, fiberwise Lyapunov exponent $\lambda^F_i$ for $\mu$ then,  by  dimension counting,   there are at most $r(G)-1$ coarse resonant roots $\xi \in \hat \Sigma$.  However, as $\mu$ is $P$-invariant and as  there are no proper parabolic subalgebras of resonant codimension smaller that $r(G)$, it follows that $\mu$ is necessarily $G$-invariant.  It then follows that if $\td \alpha$ has no invariant probability measure on $M^\alpha$, then every fiberwise Lyapunov exponent of $\mu$ is resonant with a root of $\lieg$.  We claim in this case that every fiberwise exponent  for $\mu$ is in fact resonant with a negative root $\beta\in \Sigma_-$.  Indeed, if there existed  a fiberwise exponent that was resonant with a positive root, then would be at most $r(G)-1$ resonant negative roots.  As we assume $\mu$ is $P$-invariant we again generate a parabolic subgroup which preserves $\mu$ and with resonant codimension smaller than $r(G)$.  This again implies the existence of an
 $\td \alpha$ invariant probability on $M^\alpha$.   

Thus, in the case  that $\dim(M)= r(G)$ it follows that   if there is no $\alpha$-invariant measure on $M$ then there exists   $s\in A$ such  that $\lambda^F_i(s)<0$ for every fiberwise Lyapunov exponent $\lambda_i^F$ of $\mu$.  Proposition \ref{thm:resonant} then holds and a standard argument shows in this case that the fiberwise conditional measures $\mu_g$ are supported on a finite set for almost every $g$.    By ergodicity, the number of atoms is constant \as

In the case that $\dim(M) \le m(G)$ and every fiberwise Lyapunov exponent is resonant with a negative root, the same analysis as above holds.  In particular the hypotheses of Proposition \ref{thm:resonant} hold.   Note that this holds even if $\lieg$ has rank-1 simple ideals (so $m(G) = 1$  and $M$ is a circle.)
If $\dim(M) \le m(G)$ and not every fiberwise Lyapunov exponent is resonant with a negative root, then there are at most $m(G)-1$ resonant, negative coarse  restricted roots.   Note if $\lieg$ has rank-1 simple ideals then, as $m(G)-1 = 0$, this implies $\Qin(\mu) = \Qout(\mu) = G$.  
Thus we may assume $\lieg$ has no rank-1 simple ideals.  
From the definition of $m(G)$,   it follows that either $\Qin(\mu)= G$ or that $\Qin(\mu)$ is a maximal parabolic subgroup and, from     Proposition \ref{thm:High entropy}, we have that $\Qin(\mu) = \Qout(\mu)$.

In particular, under the hypotheses of  either Theorem  \ref{thm:measconjfin} or \ref{thm:measconj},
we have $Q:= \Qin(\mu) = \Qout(\mu)$.  

 
 In the setting of either theorem, the spaces $(M, \mu_g)$ are Lebesgue probability spaces.  As there are at most countably many isomorphism types of Lebesgue probability spaces, by $P$-ergodicity it follows that the spaces $(M, \mu_g)$ are all measurably isomorphic to a fixed abstract Lebesgue probability space $(Y, \eta)$.  In particular,  there exists a measurable family of measurable isomorphisms $\phi_g\colon (M,\mu_g) \to (Y, \eta)$.  Moreover the family of isomorphism $\phi_g$ is $Q$-invariant.  The family of isomorphisms $\phi_g$ translate the $\Gamma$-equivariance of the family $\mu_g$   to a family of automorphisms of the measure space $(Y,\eta)$ parameterized by $Q\bs G$:
$$\psi(\gamma, Qg) :=\phi_{g\gamma}\circ \alpha(\gamma)_* \circ \phi_g\inv \in \aut(Y,\eta).$$
One verifies that $\psi$ is a cocycle over the right $\Gamma$-action on $Q\bs G$.  

It now follow from 
Claim \ref{prop:msblisom} that $(M,\hat \mu)$ is measurably isomorphic to $(Q\bs G, \nu)\times (Y, \eta)$.  Moreover,   the action $\alpha$ of  $\Gamma$ on $(M, \hat \mu)$ is measurably conjugate via this isomorphism to the skew action defined by $\psi$ over the standard right  action of  $\Gamma$ on $Q\bs G$.  
Theorems \ref{thm:measconjfin} and \ref{thm:measconj} now follow.
\end{proof}

\section{Proof of Propositions \ref{thm:resonant} and \ref{thm:High entropy}}
We recall the notation of Section \ref{sec:QinQout}.  In particular, we take $\mu$ to be a $P$-invariant, $P$-ergodic measure on $M^\alpha$.   
Recall also the definitions of  $\Qin(\mu)$ and $\Qout(\mu)$   in Section \ref{sec:QinQout}.  We verify under the hypotheses of 
Propositions \ref{thm:resonant} and \ref{thm:High entropy} that   $\Qin(\mu)=\Qout(\mu)$.

\subsection{Conditional measures along coarse root spaces}
Under the assumption that $\Qin(\mu)\neq\Qout(\mu)$ the following claim guarantees the existence of a coarse restricted root $\xi\in \hat \Sigma$ with $G^\xi \not\subset \Qin (\mu)$ and such that the measure $[\mu^{G^\xi}_p]$ is non-trivial. Write $Q = \Qin(\mu)$ and  $\lieq = \Lie(Q)$ for the remainder.  

\begin{claim}\label{claim:nontrivialdirection}
Suppose $\Qin(\mu)\neq\Qout(\mu)$.    Then there is a coarse  restricted root $\xi \in\hat  \Sigma$ with $\lieg^\xi \not\subset \lieq$  and such that $\mu^{G^\xi}_p $ is non-atomic for $\mu$-\ae $p\in M^\alpha$.    
\end{claim}
The claim follows from the local product structure of $A$-invariant measures on $G$-spaces demonstrated in \cite[Proposition 8.3]{MR1989231} and further  developed in \cite[Theorems 7.5, 8.4]{MR2191228}.  We sketch a short proof here for completeness.

Given standard parabolic subgroup $Q$ with Lie algebra $\lieq$,   let $$\Sigma_Q= \{ \beta \in \Sigma : \lieg ^\beta\subset \lieq\}, \quad \Sigma ^\perp _Q= \{ \beta \in \Sigma : \lieg ^\beta\not\subset \lieq\}.$$
 We have $\Sigma = \Sigma _Q \cup \Sigma _Q^\perp$ and $\Sigma_Q$ and $\Sigma_Q^\perp$ are saturated by coarse equivalence classes of restricted roots.   

\begin{proof}
Recall we write $Q= \Qin(\mu)$ and the measure $\mu^G_p$ is a $Q$- and hence $A$-invariant measure on $G$.   
Let $\lieg^\perp:= \bigoplus_{\beta\in \Sigma^\perp_Q}$.   Note that $\Sigma^\perp_Q$ consists  of negative roots.  
Let $V$ be the   analytic subgroup corresponding to $\lieg^\perp$.  
Let $C_s$ denote conjugation by $s\in A$.  We have $C_s(V) = V$ for $s\in A$.    As $\mu$ is $A$-invariant, we have for $s\in A$  that $[(C_s)_*\mu^V_{p}]= [\mu^V_{\td \alpha (s)(p)}]$ 
for almost every $p$.  

As $\Sigma^\perp_Q \subset \Sigma_-$, we may find an $s_0\in A$ and a coarse restricted root $\xi\subset \Sigma^\perp$ with 
\begin{itemize}
\item $\beta(s_0) = 0$ for $\beta\in \xi$;
\item $\beta(s_0) <0$ for all $\beta\in \Sigma^\perp_Q \sm \xi$.
\end{itemize}
Let $V'$ 
be the analytic subgroups of $V$ corresponding  to $\bigoplus_{\beta\in \Sigma^\perp_Q\sm \xi}\lieg^\beta$.  

  Suppose  first that  $\mu^V_p$ is not supported on a single $V'$-orbit for a positive measure of set of $p\in M^\alpha$.  
As $\td \alpha(s_0)$ acts as the identity   on $G^\xi$-orbits, we have $\mu^{G^\xi}_{\td \alpha (s_0)(p)}= \mu^{G^\xi}_{p}$ for almost every $p$.  Moreover, as $\td \alpha(s_0) $ contracts $V'$ orbits, it follows from Poincar\'e recurrence and Lusin's theorem that    $\mu^{G^\xi}_{p} =  \mu^{G^\xi}_{\td \alpha (v) (p)}$ for $\mu^{V'}_p$-\ae $v\in V'$ and $\mu$-\ae $p$.  
Thus, $\mu^{G^\xi}_p$ is atomic on a positive measure set of $p$ only if    $\mu^V_p$ is  supported on a single $V'$-orbit for a positive measure of set of $p\in M^\alpha$.   Thus, $\mu^{G^\xi}_p$ is non-atomic on a positive measure set of $p$.  
Note that the actions by $A$ and $M$ preserve the  coarse root subgroups $G^\xi$ and also preserve  the measure $\mu$. Also, as the $A$-ergodic components of $\mu$ are $N$-invariant, it follows from $P$-ergodicity of $\mu$ that $\mu^{G^\xi}_p$ is non-atomic for almost every $p$.

If $\mu^V_p$ is supported on a single $V'$-orbit for almost every $p\in M^\alpha$, we may recursively repeat the above argument with $V$ replaced $V'$.  
\end{proof}

\subsection{Recurrence and the proof of  Proposition  \ref{thm:resonant}} \label{sec:thisone}
\def\E {\mathcal E}

We show under the assumption that every fiberwise Lyapunov exponent is resonant with a negative root, that $\Qin(\mu) = \Qout(\mu)$.  Suppose that $\Qin(\mu) \neq \Qout(\mu)$ and let $\xi$ be a coarse restricted root as in Claim \ref{claim:nontrivialdirection}.  We will show below that $\mu$ is $G^\xi$-invariant.  The contradiction completes the proof of Proposition \ref{thm:resonant}.

Recall that $G = C\times G'$ 
 where $C$ is the maximal compact factor.  
  Let $( M^\alpha)'$ denote the quotient of $M^\alpha$ under the action of $\td \alpha(C)$.  Note that  $( M^\alpha)'$ might have an orbifold structure but still has the structure of a fiber-bundle (with orbifold fibers) over $ G'/\Gamma =C\bs G/\Gamma$. 
We have $C\subset M \subset P$.  As $C$ is contained in the centralizer of both $A$ and $N$,  the actions of $A$ and $N$ on $G/\Gamma$ and $M^\alpha$ descend to actions of $A$ on $C\bs G/\Gamma$ and $( M^\alpha)'$.  

Let $A'\subset A$ denote the kernel of $\xi$; that is, $s\in A'$ if $\beta(s)=0$ for all $\beta\in \xi$.  
As we assume $\Gamma$ has dense image in every rank-$1$ almost-simple subgroup of $G$, it follows from   Moore's ergodicity theorem   (applied to each irreducible factor)  that $A'$ acts ergodically on $ G'/\Gamma$ (see for example \cite[Theorem 2.2.6]{MR776417}).

Let $ \mu'$ denote the projection of the measure $\mu$ onto $   (M^\alpha)'$ and let $\{  \mu_{ g'}'\colon  g '\Gamma  \in   G'/\Gamma\}$ denote a family of conditional measures induced by the partition of $(M^\alpha)'$ into its fibers over $G'/\Gamma$.  As discussed in the proof of Theorem \ref{thm:measconjfin}, the assumption that every fiberwise Lyapunov exponent is resonant with a negative root combined with the  $C$-invariance of $\mu$ implies that $ \mu'_{ g'}$ has  finite support for almost every $ g'\Gamma \in G'/\Gamma$.  Moreover, $P$-ergodicity of $\mu$ ensures that the number of atoms is constant in $ g'\Gamma$.

Note that (as we assume  $\mu^{G^\xi}_p$ is non-atomic) the partition   of $(( M^\alpha)',  \mu')$ into full $G^\xi$-orbits is non-measurable.  
Let $ \eta^\xi$ denote the measurable hull of this partition; that is the finest measurable partition of $(( M^\alpha)',   \mu')$ containing full $G^\xi$-orbits.

Consider the action of $A'$ on $(( M^\alpha)',  \mu')$.   Note that the action   need not be ergodic.  Let $ \E_{A'}$ denote the partition into ergodic components  of $ \mu'$ with respect to the action of $A'$. We have the following claim which will provide the  necessary recurrence to complete the  proof of Proposition \ref{thm:resonant}.

\begin{claim}\label{pi-part}
The partition $ \eta^\xi$ refines $\E_{A'}$.  
\end{claim}
\begin{proof}

	Let $\E_A$ denote the partition into  ergodic components of $ \mu'$ with respect to the action of $A$ on $( M^\alpha)'$.  Taking $s\in A$ such that $\xi(s)<0$, it follows that the partition of $(M^\alpha)'$ into $G^\xi$-orbits defines a (uniformly) contracting foliation of $(M^\alpha)'$ under the action $\td \alpha (s)$.  
By the pointwise ergodic theorem, it follows  that the partition of $(( M^\alpha)',  \mu') $ into $\td \alpha(s)$-ergodic components is refined by $\eta^\xi$.  To complete the proof of the claim we show  $\E_A = \E_{A'}$.   

	
Fix $ p'\in ( M^\alpha)'$ and let $ (\mu')^{\E_A}_{ p'}$ be the $A$-ergodic component of $ \mu'$ containing $ p'$.  
Let $\td  \E( p')$ denote the partition into $A'$-ergodic components of $(( M^\alpha)',  (\mu')^{ \E_A}_{ p'})$.  
We claim that $\td  \E( p')$ is finite for almost every $ p'$.  
	Indeed first note that, as both $A$ and $A'$ act ergodically on $G'/\Gamma$, the $A$ and $A'$-ergodic components of $(( M^\alpha)',  \mu')$ project to the Haar measure on $G'/\Gamma.$
	Furthermore, as the fiber conditional measures $( \mu')_{ g' \Gamma}$ are purely atomic 
	and as the ergodic components of the $A'$-action on $(( M^\alpha)', ( \mu')^{ \E_A}_{ p'})$  
	are mutually singular,  it follows that the partition $\td  \E( p')$ is finite. 

 As $A'\subset A$ with $A$ abelian, it follows that $A$ permutes elements of the partition $\td \E( p')$ of $(( M^\alpha)',  (\mu')^{ \E_A}_{ p'})$.    Note that the partition $\td \E( p')$ is finite, $A$ acts ergodically on $(( M^\alpha)', ( \mu')^{\E_A}_{ p'})$, and   $A$ is a connected group.   In particular, $A$ acts ergodically on the (finite) factor measure space $(( M^\alpha)', ( \mu')^{\E_A}_{ p'})/ \td \E( p')$.   This yields a contradiction unless  the partition $\td \E( p')$ contains only one element.  
 It follows that $ \E_A =  \E_{A'}$.   
\end{proof}

Now consider $\td \eta^\xi$, the measurable hull of the partition into $G^\xi$-orbits on $(M^\alpha, \mu)$ and $\td \E_{A'}$, the partition of $(M^\alpha, \mu)$ into $A'$-ergodic components.  We claim $\td \eta^\xi$ refines $\td  \E_{A'}$.  Indeed, since $C$ centralizes  $A'$, $C$ permutes  elements of $\td  \E_{A'}$.  Similarly, $C$ centralizes  $G^\xi$ and hence permutes elements of $\td \eta^\xi$.  Moreover,  $[\mu^{G^\xi}_p] = [\mu^{G^\xi}_{\td \alpha(g)(p)}] $ for $g\in C$.   
If  elements of $\td  \E_{A'}$ did not contain full $G^\xi$-orbits modulo $\mu$, then the $C$-orbits of elements of  $\td \E_{A'}$ would not  contain $C$-orbits of full $G^\xi$-orbits modulo $\mu$ contradicting the above claim.  

We now consider the $A'$-action on $M^\alpha$.  As $A'$ acts isometrically on $G^\xi$-orbits and as $\td \eta^\xi$ refines $\td  \E_{A'}$, from  standard   measure rigidity  arguments  for actions  of Abelian groups we obtain the following.    

\begin{claim}\label{clm:43}$\mu$ is invariant under the  action   $G^\xi$.  \end{claim}
We only outline the main steps in the proof of Claim \ref{clm:43}.

\begin{proof}[Proof of Claim \ref{clm:43}]
Fix  $U\subset G^\xi$  a pre-compact, open neighborhood of the identity in $G^\xi$.  Given almost every $p\in M^\alpha$,  the measure $\mu^{G^\xi}_p$ gives positive mass to $U$.  For such $p$,  normalize $\mu^{G^\xi}_p$ on $U$.

Let $A'$ be as above.  Then any $s\in A'$ commutes with $G^\xi$ whence 
 $C_s(U) := sUs\inv= U$ and  $\mu^{G^\xi}_{\td \alpha(s)(p)} = \mu^{G^\xi}_{p}$.  
Let $K\subset M^\alpha$ be a compact set on which the assignment $p\mapsto \mu^\xi_p$ is continuous (where locally finite measure on $G^\xi$ are endowed with the topology dual to compactly supported continuous functions.)   

Consider a generic $p\in K$.  Recall that  the $\alpha(G^\xi)$-orbit of $p$ is contained in the $A'$-ergodic component of $\mu$  containing $p$.  Consider any $p'\in \td \alpha(G^\xi)(p)\cap K$   that is a density point of $K$ with respect to the $A'$-ergodic component of $\mu$ containing $p$.   
It follows that there is a sequence $s_k\in A'$ with 
\begin{enumerate}
\item $\td \alpha(s_k)(p)\in K$ for every $k\in \N$;
\item $\td \alpha(s_k)(p)\to p'$ as $k\to \infty$;
\item   $\mu^\xi_p = \mu^\xi_{\td \alpha(s_k)(p)}$ for every $k\in \N$. 
\end{enumerate}
It follows that $\mu^\xi_p = \mu^\xi_{p'}$. Taking sets $K_j$ as above of measure arbitrarily close to 1, for typical points $p$, it follows that $\mu^\xi_p= \mu^\xi_{p'}$ for all $p'= \td \alpha(v)(p)$ for a $\mu^{G^\xi}_p$-conull set of $v$.  
It follows that for almost every $p$, the group of isometries of  $G^\xi$ preserving $\mu^{G^\xi}_p$ up to normalization acts transitively on the support of $\mu^{G^\xi}_p$ in $G^\xi.$  In fact, the group of  {right}-translations of  $G^\xi$ preserving $\mu^{G^\xi}_p$ up to normalization acts transitively on the support of $\mu^{G^\xi}_p$ in $G^\xi.$

It now follows from arguments developed in  \cite[Section 5]{MR1406432}   that $[\mu^{G^\xi}_p]$ coincides with the Haar measure on  a non-trivial subgroup $V(p)\subset G^\xi$.  See also \cite[Section 6.1]{MR2122918} for an argument in the framework described here.  
Moreover, the assignment $p\mapsto V(p)$ is measurable and constant on $A$-orbits.  From  Proposition \ref{prop:aaalll} it follows that $\mu$ is $G^\xi$-invariant.  \end{proof}

Recall our initial choice of $\xi$ was such that  $G^\xi\not\subset Q= \Qin(\mu)$.  From this contradiction we conclude that $\Qin(\mu) = \Qout(\mu)$.   This completes the proof of   Proposition \ref{thm:resonant}

\subsection{Proof of  Proposition \ref{thm:High entropy}} 
The proof of Proposition \ref{thm:High entropy} is a direct application of the Theorem \ref{thm:HE}.  Recall the definitions of $\Sigma_Q$ and $ \Sigma ^\perp _Q$ above.  
Suppose that $\Qin(\mu)\neq \Qout(\mu)$.  Let $\xi \in \hat \Sigma$ be as in Claim \ref{claim:nontrivialdirection}.  Then $\xi \subset \Sigma ^\perp_{\Qin(\mu)}$. 
Write $Q = \Qin(\mu)$.

If $\xi$ contains two elements, we have $\xi = \{\beta', 2 \beta'\}$ for some   root $\beta'\in \Sigma_-$.  In the latter case, take $\beta = 2 \beta'$ if $[\mu^{G^\xi}_p]$ is supported on $G^{2\beta'}$ for almost every $p$ and $\beta = \beta'$ otherwise.  If $\xi$ is a single root take $\beta$ with $\xi= \{\beta\}$.

We claim
\begin{claim}
If $\Qin(\mu)$ is maximal  then, with  $\beta$ as above, there is a nonzero root $\gamma \in  \Sigma_{\Qin(\mu)}$ with  
\begin{enumerate}
	\item $\gamma\neq -c \beta$, for any $c>0$; 
	\item $\gamma+ \beta \in \Sigma$;   
	\item $\gamma+ \beta \in \Sigma^\perp_{\Qin(\mu)}.$
\end{enumerate} 
\end{claim}
\begin{proof}
Indeed let $\lieq= \Lie(\Qin(\mu))$.   Then $\lieq = \lieq_{\Pi \sm \{\alpha\}}$ for some simple root $\alpha$.  If $\beta = -\alpha$ then, as we assume the are no rank-1 simple ideals,  there is   simple positive root $\hat \alpha\neq - \alpha$ adjacent to $\alpha$ in the Dynkin diagram corresponding to the simple factor containing $\alpha$.  Then  $\hat \alpha - \beta = \hat \alpha + \alpha $ is a root.  Take $\gamma = - \hat \alpha$.  Then (since  $\lieq$ is of the form $  \lieq_{\Pi \sm \{\alpha\}}$) $\gamma =-\hat \alpha \in  \Sigma_{\Qin(\mu)}$ and 
 $\gamma+ \beta \in \Sigma^\perp_{\Qin(\mu)}.$  Similarly, if $\beta = -2\alpha$ (so that $\beta$ is a root in factor of type $BC_n$) then $\alpha $ is the right-most root in the Dynkin diagram; with $\hat \alpha$ the root adjacent  to (that is, to the left of) $\alpha$, since $\hat \alpha + 2 \alpha$ is a root, $\gamma=-\hat \alpha$ satisfies the conclusions of the claim.

If $\beta\neq -\alpha$ and $\beta\neq -2\alpha$  then $\beta$ is of the form $$\beta = c_\alpha \alpha + \sum _{\hat \alpha \neq \alpha\in \Pi} c_{\hat \alpha} \hat \alpha$$
where $c_\alpha <0$,  $c_{\hat \alpha} \le 0$, and $  \sum _{\hat \alpha \neq \alpha\in \Pi} c_{\hat \alpha} \ge 1.$
 Since $\beta$ is not a simple negative root,  there is a simple (positive) root $\alpha'\in \Pi$ such that $\beta + \alpha'$ is a negative root.  If $\alpha'\neq \alpha$ then, since $\beta=  (\beta + \alpha') - \alpha'$ and $-\alpha' \in \Sigma_{\Qin(\mu)}$, 
 it follows that $$(\beta + \alpha') \notin \Sigma_{\Qin(\mu)}$$ since $\Qin(\mu)$ is a subgroup.  Then $\gamma= \alpha'$ satisfies the conclusion of the claim.
 On the other hand, if $\alpha'= \alpha$ then, since $\beta + \alpha $ is a negative root, $$ -(\beta + \alpha) \in \Sigma_+\subset  \Sigma_{\Qin(\mu)}$$ and 
 $$\beta+ -(\beta + \alpha) = -\alpha  \notin \Sigma_{\Qin(\mu)}.$$ Since $\beta$ is linearly independent of $\alpha$,  the root  $\gamma= -(\beta + \alpha)$ satisfies the conclusion of the claim. 
\end{proof}

As we assume that $\gamma \neq-c\beta $ for $c>0$, it follows that we may find $s\in A$ with $\beta(s)<0$ and $\gamma(s)<0$.  
Let $\lieh$ be the Lie subalgebra generated by $\lieg^\xi\oplus \lieg^{[\gamma]}$ where $[\gamma]$ is the coarse equivalence class of $\gamma$.   
Then $H= \exp(\lieh)$ is the minimal subgroup containing $G^\xi$ and  $G^{[\gamma]}$ that is contracted by all $s$  with $\beta(s)<0$ and $\gamma(s)<0$.
For  $p\in M^\alpha$, let $H_p\subset S_p\subset H$  be the subgroups guaranteed by  Theorem \ref{thm:HE}.

Let $\bar \beta = \beta +\gamma $. As $\bar \beta \in \Sigma$, from our choice of $\beta$ and that $\gamma\in \Sigma_Q$ there are 
$g\in G^\beta \cap S_p$ and $h\in G^{\gamma}\cap S_p$ that do not commute.  Theorem \ref{thm:HE} 
implies that $\lie(H_p)$  contains a nontrivial intersection with $\lieg^{\bar \beta} = [\lieg^\beta , \lieg^\gamma]$. 
In particular, one can find a measurable $A$-invariant family of subgroups $V(x) \subset G^{[\bar \beta]}$ satisfying the hypotheses of Lemma \ref{prop:aaalll}.
From Lemma \ref{prop:aaalll}, it follows that  the measure $\mu$ is $G^{[\bar \beta]}$-invariant contradicting the choice of $\gamma$.  
This completes the proof of Proposition \ref{thm:High entropy}.

\newpage 
\appendix
\section{Tables of root data for classical root systems} 
A table of simple roots and all positive roots is given in Table \ref{table:1}.  We express the roots in terms of a standard presentation (c.f. \cite{MR1920389}[Appendix C].)
In all cases, the parabolic subalgebra $\lieq$ of minimal resonant codimension is  $\lieq = \lieq_{\Pi\sm \{\alpha_1\}}$ from which we immediately verify   $r(\lieg)$ in Example \ref{ex:1} from Table \ref{table:1}.
We also verify that  $m(\lieg)$ is  the  resonant codimension of $\lieq = \lieq_{\Pi\sm \{\alpha_1, \alpha_2\}}$ except for $D_4$ from which we verify $m(\lieg)$ in  Examples \ref{ex:2}. 

\DynkTablePosRoots

\bibliographystyle{AWBmath}

\bibliography{bibliography}

\end{document}